\documentclass[12pt, reqno]{amsart}
\usepackage{amsmath, amsthm, amscd, amsfonts, amssymb, graphicx, color}
\usepackage[bookmarksnumbered, colorlinks, plainpages]{hyperref}
\hypersetup{colorlinks=true,linkcolor=red, anchorcolor=green, citecolor=cyan, urlcolor=red, filecolor=magenta, pdftoolbar=true}

\newtheorem{theorem}{Theorem}[section]
\newtheorem{lemma}[theorem]{Lemma}
\newtheorem{proposition}[theorem]{Proposition}
\newtheorem{corollary}[theorem]{Corollary}

\theoremstyle{definition}
\newtheorem{definition}[theorem]{Definition}

\theoremstyle{remark}
\newtheorem{remark}[theorem]{Remark}

\numberwithin{equation}{section}

\newcommand\C{\mathbb C}
\newcommand\D{\mathbb D}

\newcommand\rikmu{R^\infty(K,\mu)}

\newcommand\limu{L^\infty(\mu)}
\renewcommand\i{\infty}
\newcommand\area{{\frak m}}
\newcommand\CT{{\mathcal C}}

\newcommand{\quotes}[1]{``#1''}

\begin{document}

\setcounter{page}{1}

\title[Weak-Star Closed Algebras]{Invertibility in Weak-Star Closed Algebras of Analytic Functions}

\author[L. Yang]{Liming Yang$^1$}

\address{$^1$Department of Mathematics, Virginia Polytechnic and State University, Blacksburg, VA 24061.}
\email{\textcolor[rgb]{0.00,0.00,0.84}{yliming@vt.edu}}




\begin{abstract} For $K\subset \C$ a compact subset and $\mu$ a positive finite Borel measure supported on $K,$ let $R^\i (K,\mu)$ be the weak-star closure in $L^\i(\mu)$ of rational functions with poles off $K.$ We show that if $R^\i (K,\mu)$ has no non-trivial  $L^\i$ summands and $f\in R^\i (K,\mu),$ then $f$ is invertible in $R^\i (K,\mu)$ if and only if Chaumat's map for $K$ and $\mu$ applied to $f$ is bounded away from zero on the envelope with respect to $K$ and $\mu.$ The result proves the conjecture $\diamond$ posed by J. Dudziak \cite{dud84} in 1984.
\end{abstract} \maketitle

\section{\textbf{Introduction}}

For a Borel subset $B$ of the complex plane $\C,$ let $M_0(B)$ denote the set of finite complex-valued Borel measures that are compactly supported in $B$ and let $M_0^+(B)$ be the set of positive measures in $M_0(B).$ The support of $\nu\in M_0(\C),$ $\text{spt}(\nu),$ is the smallest closed set that has full $|\nu|$ measure. For a Borel set $A\subset \C,$ $\nu_A$ denotes $\nu$ restricted to $A.$

For a compact subset $K\subset \C$ and $\mu\in M_0^+(K),$ the functions in 
$\mbox{Rat}(K) := \{q:\mbox{$q$ is a rational function 
with poles off $K$}\}$
are members of $L^\i (\mu)$. We let $\rikmu$ be the weak-star closure of $\mbox{Rat}(K)$ in $\limu.$ $C(K)$ denotes the space of all continuous functions on $K$ while $R(K)$ denotes the uniform closure in $C(K)$ of $\mbox{Rat}(K).$ Let $R(K)^\perp$ be the subset of measures $\nu \in M_0(K)$ such that $\int f d \nu = 0$ for each $f \in R(K).$ Then $R(K)^\perp \cap L ^1(\mu)$ is the space of weak-star continuous annihilators of $\rikmu.$ 

There exists a Borel partition $\{\Delta_0,\Delta_1\}$ of $K$ such that 
\begin{eqnarray}\label{RIDecomp}
\ \rikmu = L^\i (\mu|_{\Delta_0}) \oplus R^\i (K,\mu|_{\Delta_1}) 
\end{eqnarray}
where $R^\i (K,\mu|_{\Delta_1})$ contains no nontrivial $L^\i$ summands (see \cite[Proposition 1.16 on page 281]{conway} ).
Call $\rikmu$ pure if $\Delta_0 = \emptyset$ in \eqref{RIDecomp}. Because of \eqref{RIDecomp}, we shall assume that $\rikmu$ is pure in this paper.

The {\em envelope} $E$ with respect to $K$ and $\mu\in M_0^+(K)$ is the set of points $\lambda \in  K$ such that there exists $\mu_\lambda\in M_0(K)$ that is absolutely continuous with respect to $\mu$ such that $\mu_\lambda(\{\lambda\}) = 0$ and $\int f(z) d\mu_\lambda(z) = f(\lambda)$ for each $f \in R(K).$ For $\lambda\in E$ and $f\in R^\i (K,,\mu),$ set $\rho(f)(\lambda) = \int fd\mu_\lambda.$ Clearly $\rho(f)(\lambda)$ is independent of the particular $\mu_\lambda$ chosen. We thus have a map $\rho$ called Chaumat's map for $K$ and $\mu,$ which associates
to each function in $\rikmu$ a point function on $E.$ Chaumat's Theorem \cite{cha74} states the following: The map $\rho$ is an isometric isomorphism and a weak-star homeomorphism from $\rikmu$ onto $R^\i (\overline E, \area_E),$ where $\area_E$ is the area measure (Lebesgue measure on $\C$) $\area$ restricted to $E$ (also see \cite[Chaumat's Theorem on page 288]{conway}) and $\overline {B}$ denotes the closure of the set $B\subset \C$. 

Our main theorem below proves the conjecture $\diamond$ posed by J. Dudziak in \cite[on page 386]{dud84}.

\begin{theorem}\label{InvTheorem}
Let $K\subset \C$ be a compact subset and $\mu \in M_0^+(K).$ Suppose that $\rikmu$ is pure, $E$ is the envelope for $K$ and $\mu,$ and $\rho$ is the Chaumat's map for $K$ and $\mu .$ If $f\in \rikmu,$ then $f$ is invertible in $\rikmu$ if and only if there exists $\epsilon_f > 0$ such that 
\[
\ |\rho(f)(z)| \ge \epsilon_f,~ \area_E-a.a..
\] 	
\end{theorem}

For $\mathcal H$ a complex separable Hilbert space, $\mathcal L(\mathcal H)$ denotes the space of all bounded linear operators on $\mathcal H.$ The spectrum and essential spectrum of an operator $A\in \mathcal L(\mathcal H)$ are denoted $\sigma(A)$ and $\sigma_e (A)$, respectively. An operator $S\in \mathcal L(\mathcal H)$ is called {\em  subnormal} if there exists a complex separable Hilbert space $\mathcal K$ containing $\mathcal H$ and a normal operator $N\in \mathcal L(\mathcal K)$ such that $N\mathcal H \subset \mathcal H$ and $S = N|_{\mathcal H}.$ Such an $N$ is called a {\em minimal normal extension} (mne) of $S$ if the smallest closed subspace of $\mathcal K$ containing $\mathcal H$ and reducing $N$ is $\mathcal K$ itself. Any two mnes of $S$ are unitarily equivalent in a manner that fixes $S.$ 
The spectrum $\sigma (S)$ is the union of $\sigma (N)$ and some collection of bounded components of $\C\setminus \sigma (N).$ The book \cite{conway} is a good reference for basic information for subnormal operators.

For a subnormal operator $S\in \mathcal L(\mathcal H)$ with $N=mne(S),$ let $\mu$ be the scalar-valued spectral measure (svsm) for $N.$ Since $\text{spt}(\mu) = \sigma(N) \subset  \sigma (S),$ we see that $\text{Rat}(\sigma (S)) \subset L^\i(\mu).$ Therefore, for $f\in R^\i(\sigma (S), \mu),$ the normal operator $f(N)$ is well defined and $f(N) \mathcal H \subset \mathcal H.$ The operator $f(S) := f(N) |_{\mathcal H}$ defines a functional calculus for $S.$ We assume that $R^\i(\sigma (S), \mu)$ is pure. Let $E$ be the envelope for $\sigma (S)$ and $\mu$ and let $\rho$ be the Chaumat's map from $R^\i(\sigma (S), \mu)$ to $R^\i(\sigma (S), \area_E).$  Define $f(\sigma (S))$ the closure of $\rho(f)(E).$ We define $cl(f,\sigma_e (S)),$ the cluster set of $f$ on $\sigma_e (S),$ to be the subset of $\lambda \in \C$ such that there exists a sequence $\{z_n\} \subset E$ and a $z_0 \in \sigma_e (S)$ satisfying $z_n\rightarrow z_0$ and $\rho(f)(z_n)\rightarrow \lambda.$  

Combining Theorem \ref{InvTheorem} with \cite[Corollary IV.6]{dud84}, we obtain the following spectral mapping theorem for subnormal operators.

\begin{corollary}\label{SMForSubnormal}
Let $S\in \mathcal L(\mathcal H)$ be a subnormal operator and let $\mu$ be the scalar-valued spectral measure for $N=mne(S).$ If 	$R^\i(\sigma (S), \mu)$ is pure and $f\in R^\i(\sigma (S), \mu),$ then $\sigma (f(S)) \subset f(\sigma (S))$ and $\sigma_e (f(S)) \subset cl(f,\sigma_e (S)).$
\end{corollary}

In section 2, we review some results of analytic capacity and Cauchy transform that are needed in our analysis. Section 3 reviews the modified Vitushkin approximation scheme of Paramonov and proves some technical lemmas. In section 4, we introduce an algebra of \quotes{bounded and analytic functions} on a \quotes{nearly open} subset. We obtain a criterion for functions in the algebra that will be used in proving Theorem \ref{InvTheorem}. In section 5, we introduce the concepts of non-removable boundary and removable set that are more appropriate than that of the envelope in studying $\rikmu.$ We obtain a criterion for functions in $\rikmu$ in section 6. We then prove Theorem \ref{InvTheorem} in section 7.

\section{\textbf{Preliminaries}}

If $B \subset\C$ is a compact subset, then  we
define the analytic capacity of $B$ by
\[
\ \gamma(B) = \sup |f'(\infty)|,
\]
where the supremum is taken over all those functions $f$ that are analytic in $\mathbb C_{\infty} \setminus B,$ where $\mathbb C_{\infty} = \mathbb C \cup \{\infty\},$ such that
$|f(z)| \le 1$ for all $z \in \mathbb{C}_\infty \setminus B$; and
$f'(\infty) := \lim _{z \rightarrow \infty} z[f(z) - f(\infty)].$
The analytic capacity of a general subset $F$ of $\mathbb{C}$ is given by: 
\[
\ \gamma (F) = \sup \{\gamma (B) : B\subset F \text{ compact}\}.
\]
Let $\area$ denote the Lebesgue measure on the complex plane (the area measure). The following elementary property can be found in \cite[Theorem VIII.2.3]{gamelin},
\begin{eqnarray}\label{AreaGammaEq}
\ \area(F) \le 4\pi \gamma(F)^2.	
\end{eqnarray}

We will write $\gamma-a.a.$, for a property that holds everywhere, except possibly on a set of analytic capacity zero. Good sources for basic information about analytic
capacity are Chapter VIII of \cite{gamelin}, Chapter V of \cite{conway}, \cite{dud10}, and \cite{Tol14}.

For $\nu \in M_0(\C)$ and $\epsilon > 0,$ $\mathcal C_\epsilon(\nu)$ is defined by
\[
\ \mathcal C_\epsilon(\nu)(z) = \int _{|w-z| > \epsilon}\dfrac{1}{w - z} d\nu (w).
\] 
The (principal value) Cauchy transform
of $\nu$ is defined by
\ \begin{eqnarray}\label{CTDefinition}
\ \mathcal C(\nu)(z) = \lim_{\epsilon \rightarrow 0} \mathcal C_\epsilon(\nu)(z)
\ \end{eqnarray}
for all $z\in\mathbb{C}$ for which the limit exists.
If $\lambda \in \mathbb{C}$ and 
\begin{eqnarray}\label{CTAIDefinition}
\ \tilde \nu (\lambda) := \int \frac{d|\nu |}{|z - \lambda |} < \infty,	
\end{eqnarray}
 then 
$\lim_{\epsilon \rightarrow 0} \mathcal C_{\epsilon}(\nu)(\lambda )$ exists. Therefore, a standard application of Fubini's
Theorem shows that $\mathcal C(\nu) \in L^s_{\mbox{loc}}(\mathbb{C})$, for $ 0 < s < 2$. In particular, it is
defined for $\area-a.a.,$ and clearly $\mathcal{C}(\nu)$ is analytic 
in $\mathbb{C}_\infty \setminus\mbox{spt}(\nu).$ In fact, from 
Corollary \ref{ZeroAC} below, we see that \eqref{CTDefinition} is defined for $\gamma-a.a..$ Throughout this paper, the Cauchy transform of a measure always means the principal value of the transform.
In the sense of distributions,
 \begin{eqnarray}\label{CTDistributionEq}
 \ \bar \partial \mathcal C(\nu) = - \pi \nu.
 \end{eqnarray}	

The maximal Cauchy transform is defined by
 \[
 \ \mathcal C_*(\nu)(z) = \sup _{\epsilon > 0}| \mathcal C_\epsilon(\nu)(z) |.
 \]

A related capacity, $\gamma _+,$ is defined for subsets $F$ of $\mathbb{C}$ by:
\[
\ \gamma_+(F) = \sup \|\mu \|,
\]
where the supremum is taken over $\mu\in M_0^+(F)$ for which $\|\mathcal{C}(\mu) \|_{L^\infty (\mathbb{C})} \le 1.$ 
Since $\mathcal C\mu$ is analytic in $\mathbb{C}_\infty \setminus \mbox{spt}(\mu)$ and $|(\mathcal{C}(\mu)'(\infty)| = \|\mu \|$, 
we have:
$\gamma _+(F) \le \gamma (F)$
for all subsets $F$ of $\mathbb{C}$.  

X. Tolsa has established the following astounding results. See \cite{Tol03} (also Theorem 6.1 and Corollary 6.3 in \cite{Tol14}) for (1) and (2). See \cite[Proposition 2.1]{Tol02} (also  \cite[Proposition 4.16]{Tol14}) for (3).

\begin{theorem}\label{TolsaTheorem}
(Tolsa 2003)
(1) $\gamma_+$ and $\gamma$ are actually equivalent. 
That is, there is an absolute constant $A_T$ such that, for all $F \subset \mathbb{C},$ 
\[
\ \gamma (F) \le A_ T \gamma_+(F).
\] 

(2) Semiadditivity of analytic capacity:
\[
\ \gamma \left (\bigcup_{i = 1}^m F_i \right ) \le A_T \sum_{i=1}^m \gamma(F_i)
\]
where $F_1,F_2,...,F_m \subset \mathbb{C}$ ($m$ may be $\i$).

(3) There is an absolute constant $C_T>0$ such that, for $a > 0$, we have:  
\[
\ \gamma(\{\mathcal{C}_*(\nu)  \geq a\}) \le \dfrac{C_T}{a} \|\nu \|.
\]

\end{theorem}

Given three distinct points $x, y, z \in \mathbb C$,
let $R(x, y, z)$ be the radius of the circle passing through $x, y,$ and $z$.
The Menger curvature is defined by
$c(x, y, z) = \frac{1}{R(x, y, z)}.$
If two or three of the points coincide, we define $c(x, y, z) = 0.$

For $\eta \in M_0^+(\mathbb C)$, define
 the curvature of $\eta$ as
 \[
 \ c^2(\eta ) = \int \int\int c(x,y,z)^2 d\eta(x)d\eta(y)d\eta(z)
 \]
 and
\[
\ c^2_\epsilon (\eta ) = \int \int\int _{\underset{\underset{|x-y|>\epsilon}{|z-x|>\epsilon}}{|y-z|>\epsilon}} c(x,y,z)^2 d\eta(x)d\eta(y)d\eta(z).
\]
 Define 
\[
\ N_2(\eta) = \sup_{\epsilon > 0}\sup_{\|f\|_{L^2(\eta)} = 1}\|\mathcal C_\epsilon (f \eta )\|_{L^2(\eta)}.
\]

A measure $\eta \in M_0^+(\C)$ is $a$-linear growth if
$\eta(\D(\lambda, \delta)) \le a \delta$ for $\lambda\in \C \text{ and }\delta > 0,$ where $\D(\lambda, \delta) := \{z:~|z - \lambda| < \delta\}.$ If $\eta \in M_0^+(\C)$ is $a$-linear growth, then there is an absolute constant $C > 0$ so that
\begin{eqnarray}\label{CTC2Eq}
\left|  \|\CT_\epsilon \eta\|_{L^2(\eta)}^2 - \dfrac 16 c^2_\epsilon (\eta ) \right|
\ \leq C \, a^2 \| \eta \|
\end{eqnarray}
(see \cite[Proposition 3.3]{Tol14}). We use $C_1,C_2,...$ for absolute constants that may change from one step to the next.

\begin{proposition} \label{GammaPlusThm}
If $F\subset \mathbb C$ is a compact subset and $\eta\in M_0^+(F),$ then the following properties are true.
\newline
(1) If $\|\CT \eta\|_{L^\i(\C)} \le 1,$ then $\eta$ is $1$-linear growth and $\|\mathcal C_\epsilon (\eta )\|_{L^\i(\C)} \le C_1$ for all $\epsilon > 0.$
\newline
(2) If $\eta$ is $1$-linear growth and $\|\mathcal C_\epsilon (\eta )\|_{L^\i(\C)} \le 1$ for all $\epsilon > 0,$ then $c^2(\eta) \le C_2\eta(F).$
\newline
(3) If $\eta$ is $1$-linear growth and $c^2(\eta) \le \eta(F),$ then there exists a subset $A\subset F$ such that $\eta (F) \le 2 \eta (A)$ and $N_2(\eta|_A) \le C_3.$
\newline
(4) If $N_2(\eta) \le 1,$ then there exists some function $w$ supported on $F$, with $0\le w \le 1$ such that
$\eta (F) \ \le\  2 \int w d\eta$
 and 
$\|\mathcal C _\epsilon (w\eta)\|_{ L^\infty (\mathbb C)} \ \le\  C_4$
for all $\epsilon > 0.$
\end{proposition}

See the proof of \cite[Theorem 4.14]{Tol14} on page 113 and page 114 for (1) and (3). (2) follows from \eqref{CTC2Eq}. (4) follows from \cite[Lemma 4.7]{Tol14}.

Combining Theorem \ref{TolsaTheorem} (1), Proposition \ref{GammaPlusThm}, and \cite{Tol03} (or \cite[Theorem 8.1]{Tol14}), we get the following corollary. The reader may also see \cite[Corollary 3.1]{acy19}.

\begin{corollary}\label{ZeroAC}
If $\nu\in M_0(\C),$ then there exists 
$\mathcal Q \subset \mathbb{C}$ with $\gamma(\mathcal Q) = 0$ such that $\lim_{\epsilon \rightarrow 0}\mathcal{C} _{\epsilon}(\nu)(z)$ 
exists for $z\in\mathbb{C}\setminus \mathcal Q$.
\end{corollary} 

\begin{corollary}\label{ZeroACEta}
Let $\eta\in M_0^+(\C)$ such that $\|\mathcal C (\eta)\| \le 1$. If $F$ is a compact subset and $\gamma(F) = 0$, then $\eta(F) = 0$.
\end{corollary}

\begin{proof}
Suppose $\eta(F) > 0$. By Proposition \ref{GammaPlusThm} (1) \& (2), $c^2(\eta) < \infty$. Since $c^2(\eta|_F) \le c^2(\eta),$ we see that $\gamma(F) > 0$ by Proposition \ref{GammaPlusThm} (3) \& (4), and Theorem \ref{TolsaTheorem} (1). This is a contradiction.
\end{proof}

\begin{lemma}\label{CTUniformC} Let $\mu\in M_0^+(\C).$ Suppose $f_n,f\in L^1(\mu)$ such that $\|f_n - f\|_{L^1(\mu)}\rightarrow 0$
and $ f_n \rightarrow  f, ~\mu-a.a..$
 Then, for $\epsilon > 0$, there exists a subset $A_\epsilon$ with $\gamma(A_\epsilon) < \epsilon$ and a subsequence $\{f_{n_k}\}$ so that $\{\mathcal C(f_{n_k}\mu)\}$ uniformly converges to $\mathcal C(f\mu)$ on $\mathbb C \setminus A_\epsilon$.
\end{lemma}

\begin{proof} From Corollary \ref{ZeroAC}, we let $\mathcal Q_1\subset \C$ with $\gamma(\mathcal Q_1) = 0$ such that the principal values of $\mathcal{C}(f_n\mu)(z )$ for $n\ge 1$ and $\mathcal{C}(f\mu)(z )$ exist for $z\in \C\setminus \mathcal Q_1$. Define
 \[
 \ A_{nm} = \left\{z\in \C\setminus \mathcal Q_1:~ |\mathcal{C}(f_n\mu)(z ) - \mathcal{C}(f\mu)(z )|\ge \dfrac{1}{m} \right\}.
 \]
Since
 \[
 \begin{aligned}
 \ |\mathcal{C}(f_n\mu)(z ) - \mathcal{C}(f\mu)(z )| & = \lim_{\epsilon \rightarrow 0} | \mathcal{C} _{\epsilon}(f_n\mu)(z ) - \mathcal{C} _{\epsilon}(f\mu)(z )|\\
& \le \mathcal{C} _{*}((f_n - f)\mu)(z ),
 \end{aligned}
 \]
applying Theorem \ref{TolsaTheorem} (3), we get
 \[
 \ \gamma (A_{nm}) \le \gamma \left\{\mathcal{C} _{*}((f_n - f)\mu)(z ) \ge \dfrac{1}{m} \right\} \le C_Tm\|f_n-f\|_{L^1(\mu)}.  
 \]
Choose $n_m$ so that $\|f_{n_m}-f\|_{L^1(\mu)} \le \frac{1}{m2^m}$ and we have
$\gamma (A_{n_mm}) \le \frac{C_T}{2^m}$.

Set $B_k = \cup_{m=k}^\infty A_{n_mm}$.
Applying Theorem \ref{TolsaTheorem} (2), there exists $k_0$ so that $\displaystyle \frac{A_TC_T}{2^{k_0 - 1}} < \epsilon$ and 
 \[
 \ \gamma (B_{k_0}\cup \mathcal Q_1) \le A_T\sum_{m=k}^\infty \gamma (A_{n_mm})\le A_TC_T \sum_{m=k_0}^\infty \dfrac{1}{2^m} < \epsilon.
 \]
Then on $(B_{k_0}\cup \mathcal Q_1)^c$, $ \mathcal C(f_{n_m}\mu)(z)$ converges to $\mathcal C(f\mu)(z)$ uniformly.      
\end{proof}

\begin{lemma} \label{CTMaxFunctFinite}
Let $\{\nu_j\} \subset M_0(\C).$ Then for $\epsilon > 0$, there exists a Borel subset $F$ such that $\gamma (F^c) < \epsilon$ and $\mathcal C_*(\nu_j)(z)\le M_j < \infty$ for $z \in F$.
\end{lemma}

\begin{proof}
Let $A_j = \{\mathcal C_*(\nu_j)(z) \le M_j\}$. By Theorem \ref{TolsaTheorem} (3), we can select $M_j>0$ so that $\gamma(A_j^c) < \frac{\epsilon}{2^{j+1}A_T}.$ Set $F = \cap_{j=1}^\infty A_j.$ Then applying Theorem \ref{TolsaTheorem} (2), we get
\[
 \ \gamma (F^c) \le A_T \sum_{j=1}^\infty \gamma(A_j^c) < \epsilon.
 \]
\end{proof}

\begin{lemma} \label{BBFunctLemma}
Suppose that $\{g_n\}\subset L^1(\mu)$ and $F$ is a bounded subset with $\gamma(F) > 0$. Then there exists $\eta \in M_0^+(F)$ satisfying:

(1) $\eta$ is $1$-linear growth, $\|\mathcal C_{\epsilon}(\eta)\|_{L^\infty (\mathbb C)}\le 1$ for all $\epsilon > 0,$ and 
$\gamma(F) \le C_5 \|\eta\|$;

(2) $\mathcal C_*(g_n\mu ) \in  L^\infty(\eta)$;

(3) there exists a subsequence $f_k(z) = \mathcal C_{\epsilon_k}(\eta)(z)$ such that $f_k$ converges to $f\in L^\infty(\mu)$ in weak-star topology, and $f_k(\lambda) $ converges to $f(\lambda) = \mathcal C(\eta)(\lambda)$ uniformly on any compact subset of $F^c$ as $\epsilon_k\rightarrow 0$.
Moreover, for $n \ge 1,$ 
\begin{eqnarray}\label{BBFunctLemmaEq1}
 \ \int f(z) g_n(z)d\mu (z) = - \int \mathcal C(g_n\mu) (z) d\eta (z).
 \end{eqnarray}
\end{lemma}

\begin{proof}
From Lemma \ref{CTMaxFunctFinite}, we find a compact subset $F_1\subset F$ such that $\gamma(F\setminus F_1) < \frac{\gamma(F)}{2A_T}$ and $\mathcal C_*(g_n\mu )(z) \le M_n < \infty$ for $z\in F_1$. Using Theorem \ref{TolsaTheorem} (2), we get
 $\gamma(F_1) \ge \frac{1}{A_T}\gamma(F) - \gamma(F\setminus F_1) \ge \frac{1}{2A_T}\gamma(F).$
Using Theorem \ref{TolsaTheorem} (1) and Proposition \ref{GammaPlusThm} (1),
there exists $\eta\in M_0^+(F_1)$ satisfying (1). So (2) holds.
Clearly,
 \begin{eqnarray}\label{lemmaBasicEq3}
 \  \int \mathcal C_\epsilon(\eta)(z) g_nd\mu  = - \int \mathcal C_\epsilon(g_n\mu)(z) d\eta
 \end{eqnarray}
for $n \ge 1$. We can choose a sequence $f_k(\lambda) = \mathcal C_{\epsilon_k}(\eta)(\lambda)$ that converges to $f$ in $L^\infty(\mu)$ weak-star topology and $f_k(\lambda)$ uniformly tends to $f(\lambda)$ on any compact subset of $F^c$. On the other hand, by Corollary \ref{ZeroAC} and Corollary \ref{ZeroACEta}, $|\mathcal C_{\epsilon_k}(g_n\mu)(z) | \le M _n,~ \eta-a.a.$ and  $\lim_{k\rightarrow \infty} \mathcal C_{\epsilon_k}(g_n\mu)(z)  = \mathcal C(g_n\mu)(z) ,~ \eta -a.a.$. Applying the Lebesgue dominated convergence theorem to \eqref{lemmaBasicEq3}, we get \eqref{BBFunctLemmaEq1}.
\end{proof}

For $\nu\in M_0(\C),$ define $\Theta_\nu (\lambda ) := \lim_{\delta\rightarrow 0} \frac{|\nu |(\D(\lambda , \delta ))}{\delta}$
 if the limit exists.
The following lemma follows from \cite[Lemma 8.12]{Tol14}.

\begin{lemma}\label{RNDecom2}
Let $B\subset \C$ be a bounded measurable subset and $g\in L^1(\area_B).$ Then
\[
\ \Theta_{g\area_B} (z) = 0,~\gamma-a.a..
\]
\end{lemma}

\begin{definition}\label{GLDef}
Let  $\mathcal Q$ be a set with $\gamma(\mathcal Q) = 0$.
Let $f(z)$ be a function defined 
on $\D(\lambda, \delta_0)\setminus \mathcal Q$ for some $\delta_0 > 0.$ The function $f$ has a $\gamma$-limit $a$ at $\lambda$ if
\[  
 \  \lim_{\delta \rightarrow 0} \dfrac{\gamma(\D(\lambda, \delta) \cap \{|f(z) - a| > \epsilon\})} {\delta}= 0
\]
for all $\epsilon > 0$. If in addition, $f(\lambda)$ is well defined and $a = f(\lambda)$, then $f$ is $\gamma$-continuous at $\lambda$. 
\end{definition}

The following lemma is straightforward (see \cite[Corollary 2.5]{cy22}). 

\begin{lemma}\label{GCProp}
If $f(z)$ and $g(z)$ are $\gamma$-continuous at $\lambda,$ then $f(z)+g(z)$ and $f(z)g(z)$ are $\gamma$-continuous at $\lambda.$ If in addition $g(\lambda) \ne 0,$ then $\frac{f(z)}{g(z)}$ is $\gamma$-continuous at $\lambda.$ 
 \end{lemma}

 The following lemma is from \cite[Lemma 3.2]{acy19}. 

\begin{lemma}\label{CauchyTLemma} 
Let $\nu\in M_0(\mathbb{C})$ and assume that for some $\lambda$ in $\mathbb C$ we have:
\begin{itemize}
\item[(a)] $\underset{\delta\rightarrow 0}{\overline \lim} \dfrac{| \nu |(\D (\lambda, \delta))}{\delta }= 0$ and 
\item[(b)] $\mathcal{C} (\nu)(\lambda) = \lim_{\epsilon \rightarrow 0}\mathcal{C} _{\epsilon}(\nu)(\lambda)$ exists.
\end{itemize}

Then the  Cauchy transform $\mathcal{C}(\nu)(z)$ is $\gamma$-continuous at $\lambda$.
\end{lemma}

\section{\textbf{Modified Vitushkin scheme of Paramonov and some lemmas}}

If  a compact subset $F$ is contained in $\D(a, \delta)$ and $f$ is bounded and analytic on $\C_\i \setminus F$ with $f(\infty) = 0,$ we consider the Laurent expansion of $f$ for $z\in \mathbb C\setminus \D(a, \delta),$

\[
 \ f(z) = \sum_{n=1}^\infty \dfrac{c_n(f,a)}{(z-a)^n}.
 \]
 As $c_1(f, a)$ does not depend on the choice of $a$,  we define 
 $c_1 (f) = c_1(f, a).$ The coefficient $c_2 (f,a)$ does  depend on $a$. However, if $c_1 (f) = 0$, then $c_2 (f,a)$ does not depend on $a$, and in this case, we define $c_2 (f) = c_2 (f,a).$
 
 Let $\varphi$ be a smooth function with compact support. Vitushkin's localization operator
$T_\varphi$ is defined by
 \[
 \ (T_\varphi f)(\lambda) = \dfrac{1}{\pi}\int \dfrac{f(z) - f(\lambda)}{z - \lambda} \bar\partial \varphi (z) d\area(z),
 \]
where $f\in L^1_{loc} (\C)$.	 Clearly, $(T_\varphi f)(z) = - \frac{1}{\pi}\mathcal C(\varphi \bar\partial f\area ) (z).$
Consequently, in the sense of distributions,
\[
\ \bar \partial (T_\varphi f)(z) = \varphi (z) \bar \partial f(z). 
\]
Therefore, $T_\varphi f$ is analytic outside of $\text{supp} (\bar \partial f) \cap\text{supp} (\varphi).$ If $\text{supp} (\varphi) \subset \D(a,\delta),$ then  
\[
\begin{aligned}
  \ & (T_\varphi f)(\i ) = 0,~ c_1(T_\varphi f) = - \dfrac{1}{\pi}\int f(z)\bar\partial \varphi (z) d\area(z), \\
  \ & c_2(T_\varphi f,a) =  \dfrac{1}{\pi}\int (z-a)f(z)\bar\partial \varphi (z) d\area(z).
\end{aligned}
\]
The following is the estimate of the norm of $T_\varphi f.$
\begin{eqnarray} \label{TOperProp}
 \ \| T_\varphi f\|_\i   \le  4\|f\|_\i  \delta\|\bar\partial \varphi\|.
 \end{eqnarray}	
 See \cite[VIII.7.1]{gamelin} for the details of $T_\varphi.$

Let $\delta > 0$. We say that 
$\{\varphi_{ij},S_{ij}, \delta\}$ is a smooth partition of unity subordinated to $\{2S_{ij}\},$ 
 if the following assumptions hold:

(i)  $S_{ij}$ is a square with vertices   $(i\delta,j\delta),~((i+1)\delta,j\delta),~(i\delta,(j+1)\delta),$ and $((i+1)\delta,(j+1)\delta);$

(ii) $c_{ij}$ is the center of $S_{ij}$ and $\varphi_{ij}$ is a $C^\infty$ smooth function supported in $\mathbb D(c_{ij}, \delta) \subset 2S_{ij}$ and with values in $[0,1]$;

(iii) 
 \[
 \ \|\bar\partial \varphi_{ij} \| \le \frac{C_{13}}{\delta},~ \sum \varphi_{ij} = 1.
 \]
(See \cite[VIII.7]{gamelin} for details).

Let $f\in C(\C_\i)$ (or $f\in L^\i (\C)$) with compact support.  Define $f_{ij} = T_{\varphi_{ij}}f,$ then $f_{ij} \ne 0$ for only finite many $(i,j).$ Clearly,
\[
\ f(z) = \sum_{f_{ij}\ne 0}f_{ij}(z).
\] 
The standard Vitushkin approximation scheme requires us to construct functions $a_{ij}$ such that $f_{ij} - a_{ij}$ has triple zeros at $\infty$, which requires us to estimate both $c_1 (a_{ij})$ and $c_2(a_{ij}, c_{ij})$ (see \cite[section 7 on page 209]{gamelin}). 

The main idea of P. V. Paramonov \cite{p95} is that one does not actually need to estimate
each coefficient $c_2(a_{ij}, c_{ij})$. It suffices to estimate 
the sum of coefficients $\sum_{j\in J_{il}} c_2(a_{ij}, c_{ij})$ for a special non-intersecting partition $\{J_{il}\}$ of $J_i = \{j_{min} \le j \le j_{max}\},$ where $j_{min} = \min \{j:~ f_{ij} \ne 0\}$ and $j_{max} = \max \{j:~ f_{ij} \ne 0\}.$

We now state the modified Vitushkin approximation scheme by P. V. Paramonov.
We fix a bounded Borel subset $F\subset \C.$ Let 
\begin{eqnarray}\label{AlphaIJ}
\ \alpha_{ij} = \gamma (\D(c_{ij}, k_1\delta) \cap F),
\end{eqnarray} 
where $k_1\ge 3$ is a fixed integer.
   Let $m_1 =   \min\{i,j,~  \D(c_{ij}, k_1\delta) \cap F \ne \emptyset\}$ and $m_2 = \max\{i,j,~  \D(c_{ij}, k_1\delta) \cap F \ne \emptyset\}.$
   Note both $m_1$ and $m_2$ are finite.
Set $min_i = \min\{j:~  \alpha_{ij}\ne 0 \}$ and $max_i = \max\{j:~ \alpha_{ij}\ne 0 \}$. Let $J_i = \{j:~ min_i \le j \le max_i\}$. We will also fix a constant $C_0 > 1.$

\begin{definition}
For $j_1 + 1\in J_i,$ let 
 \begin{eqnarray}\label{IndexDefinition1}
 \ J = \{j:~ j_1 + 1\le j \le j_1 + s_1 + s_2+s_3 \}.
 \end{eqnarray}
 Choose positive integers $s_1,$ $s_2$, and $s_3$ so that
 \begin{eqnarray}\label{IndexDefinition2}
 \ \delta \le \sum_{j=j_1+1}^{j_1+s_1}\alpha_{ij} < \delta + k_1\delta,
 \end{eqnarray}
\[
\ 6k_1 + 4C_0 < s_2 \le 6k_1 + 4C_0 + 1,
\]
and 
 \begin{eqnarray}\label{IndexDefinition3}
  \ \delta \le \sum_{j=j_1+s_1+s_2+1}^{j_1+s_1+s_2+s_3}\alpha_{ij} < \delta + k_1\delta
 \end{eqnarray}
 hold. With this choice, we say that $J$ is a complete group.
 \end{definition}

 We now present a detailed description of the procedure of partitioning $J_i$ into
groups. We split each $J_i$ into (finitely many) non-intersecting groups $J_{il}$, $l = 1,...,l_i$, as follows.
Starting from the lowest index $min_i$ in $J_i$ we include in $J_{i1}$ (going upwards and without jumps in $J_i$) all indices until we have collected a minimal (with respect to the number of elements) complete group $J_{i1}$. Then we repeat this
procedure for $J_i\setminus J_{i1}$, and so on. After we have constructed all the complete
groups $J_{i1},...,J_{il_i-1}$ in this way  (there may be none), then  what remains is the last
portion $J_{il_i} = J_i\setminus(J_{i1}\cup...\cup J_{il_i-1})$ of indices in $J_i$, which includes no complete
groups. We call this portion $J_{il_i}$ an incomplete group of indices (clearly, there is at most one incomplete group for each $i$).

\begin{definition}\label{MVSPDef}
Let $\delta > 0,$ $\alpha_{ij},$ and $F$ be defined as in \eqref{AlphaIJ}. 
Define
\[
\ P_1(\delta, \alpha_{ij}, F, k_1, C_0) = \{J_{il}:~1\le l \le l_i -1,~ m_1 \le i \le m_2 \}
\] 
to be the collection of all complete groups and 
\[
\ P_2(\delta, \alpha_{ij}, F, k_1, C_0) = \{J_{il_i}:~m_1 \le i \le m_2 \}
\]
 to be the set of all incomplete groups
 as
 $i$ ranges from $m_1$ to $m_2.$
   The modified Vitushkin scheme of Paramonov is defined to be 
\[
\ P(\delta, \alpha_{ij}, F, k_1, C_0) = P_1(\delta, \alpha_{ij}, F, k_1, C_0) \cup P_2(\delta, \alpha_{ij}, F, k_1, C_0).
\]
\end{definition}

For a group $J$ (complete or incomplete) with row index $i,$ 
let 
\[
\ J'(z) = \{j\in J:~\alpha_{ij} > 0 \text{ and } |z-c_{ij}| >3k_1\delta  \},
\]
and
 \[
 \ L_J' (z) = \sum_{j\in J'(z)} \left ( \dfrac{\delta\alpha_{ij}}{|z - c_{ij}|^2} + \dfrac{\delta^3}{|z - c_{ij}|^3} \right ).
 \]
Define $L_J(z) = L_J' (z)$ if $\{j\in J:~\alpha_{ij} > 0\} = J'(z)$, otherwise, $L_J(z) = 1+L_J' (z)$. For $g_{ij}$ that is bounded and analytic on $\mathbb C\setminus F_{ij},$ where $F_{ij}$ is a compact subset of $\D(c_{ij}, k_1\delta),$ define
 \begin{eqnarray}\label{GIDefinition}
 \ g_J = \sum_{j\in J} g_{ij}, ~ c_1(g_J) = \sum_{j\in J} c_1(g_{ij}),~c_2 (g_J) = \sum_{j\in J} c_2 (g_{ij},c_{ij}).
 \end{eqnarray}

\begin{definition}\label{GApplicable}
A sequence of functions 	$\{g_{ij}\}$ is applicable to $P(\delta, \alpha_{ij}, F, k_1, C_0)$ and $\mu \in M_0^+(\C)$ if the following assumptions hold, 
for some absolute constant $C_{7}$:

(1) $g_{ij}$ is bounded and analytic on $\mathbb C_\i \setminus \overline{\D(c_{ij}, k_1\delta)};$ 

(2) $g_{ij}(\infty) = c_1(g_{ij}) = 0;$ 

(3) $\|g_{ij}\|_{L^\infty(\mu)} \le C_{7}$;

(4) $|c_n(g_{ij}, c_{ij})| \le C_{7} (k_1\delta)^{n-1} \alpha_{ij}$ for $n \ge 1.$
\end{definition}

The following lemma is straightforward. 

\begin{lemma} \label{LEProp}
If $f(z)$ is bounded and analytic on $\mathbb C_\i \setminus \overline{\D(a, \delta)}$ with $f(\i) = 0,$ $\|f\| \le 1,$ $F \subset \overline{\D(a, \delta)},$ and for $n \ge 1,$
\[
 \ |c_n(f,a)| \le C_{8} n \delta ^{n-1}\gamma(F).
 \]
 Then for  $|z - a| > 2\delta,$ 
 \[
 \ \left |f(z) - \dfrac{c_1(f)}{z-a}  - \dfrac{c_2(f,a)}{(z-a)^2} \right | \le \dfrac{C_{9}\gamma(F)\delta^2}{|z-a|^3}.
 \]
 \end{lemma}
 
Therefore, if $\{g_{ij}\}$ is applicable to $P(\delta, \alpha_{ij}, F, k_1, C_0),$  then for $|z-c_{ij}| >3k_1\delta,$
 \begin{eqnarray}\label{gijEst}
 \ |g_{ij}(z)| \le  \dfrac{C_{10}\delta \alpha_{ij} }{|z-c_{ij}|^2} + \dfrac{C_{10} \delta^3 }{|z-c_{ij}|^3}.
 \end{eqnarray}

The following lemma follows easily from  Definitions \ref{MVSPDef} and \eqref{gijEst}.

\begin{lemma}\label{BasicEstimate}
Let $\{g_{ij}\}$ be applicable to $P(\delta, \alpha_{ij}, F, k_1, C_0)$ and $\mu \in M_0^+(\C)$. If $J$ is a group (complete or incomplete) with row index $i,$
 then 
 \[
 \begin{aligned}
 \ & |g_J(z)| \le C_{11} L_J(z), ~ \|g_J\|(\|g_J\|_{L^\infty(\mu)}) \le C_{11},     \\ 
 \ & c_1(g_J) = 0,~|c_2(g_J)| \le C_{11} \delta^2.
 \end{aligned}
 \]
 \end{lemma}
 
 The following key lemma  is due to \cite[Lemma 2.7]{p95}.
	
	\begin{lemma}\label{BasicEstimate3}
Let $\{g_{ij}\}$ be applicable to $P(\delta, \alpha_{ij}, F, k_1, C_0)$ and $\mu \in M_0^+(\C).$  Suppose

(H1) there exists $h_{ij} \in L^\infty (\mu)$ that is bounded and
 analytic on $\mathbb C_\i\setminus \overline{\D(c_{ij},k_1\delta)}$ 
 and satisfies
 \begin{eqnarray}\label{HConditions}
 \begin{aligned}
 \ & h_{ij}(\infty) = 0,~\|h_{ij}\| \le C_0, ~ c_1(h_{ij}) = \alpha_{ij}, \\
 \ & |c_n(h_{ij},c_{ij})| \le C_0 (k_1\delta) ^{n-1}\alpha_{ij} \text{ for } n \ge 1.
 \end{aligned}
 \end{eqnarray}
 
 Then for each complete group $J_{il}\in P_1(\delta, \alpha_{ij}, F, k_1, C_0)$, there exists a function
$h_{J_{il}}$ that has the following form
 \begin{eqnarray}\label{HFunction}
 \ \begin{aligned}
 \ & H^i_{jk} : = \dfrac{\delta}{|c_{ik} - c_{ij}|}(\lambda_{ik}^jh_{ik} - \lambda_{ij}^kh_{ij}), \\
 \ & h_{J_{il}}^0 = \sum_{j\in J_{il}^d}\sum_{k\in J_{il}^u} H^i_{jk}, \\
 \ &h_{J_{il}} = \dfrac{c_2 (g_{J_{il}})}{c_2 (h_{J_{il}}^0)}h_{J_{il}}^0, 
\ \end{aligned}
 \end{eqnarray}
where $J_{il}^d = (j_1+1,...,j_1+s_1)$ and $J_{il}^u = (j_1+s_1+s_2+1,...,j_1+s_1+s_2+s_3)$,
and there exist $\lambda_{ij}^k, \lambda_{ik}^j \ge 0$ for $j\in J_{il}^d$ and $k\in J_{il}^u$ such that (H2)-(H5) below are satisfied:

(H2) 
 \[
 \ \sum_{j\in J_{il}^d}\lambda_{ik}^j \le 1, ~ \sum_{k\in J_{il}^u} \lambda_{ij}^k \le 1;
 \]  

(H3)
 \[
 \ \sum_{j\in J_{il}^d}\sum_{k\in J_{il}^u}\lambda_{ij}^k\alpha_{ij} = \delta, ~ \sum_{j\in J_{il}^d}\sum_{k\in J_{il}^u}\lambda_{ik}^j\alpha_{ik} = \delta, ~ \lambda_{ik}^j\alpha_{ik} = \lambda_{ij}^k \alpha_{ij}; 
 \]

(H4) $c_1(H^i_{jk}) = 0,$ $s_2 > 6k_1 + 4C_0,$ and (see (2.28) in \cite{p95} for details)
	\[
	\ |c_2(H^i_{jk}) - \lambda_{ij}^k\alpha_{ij} \sqrt{-1}\delta | \le \dfrac{\lambda_{ij}^k\alpha_{ij}\delta}{2}; 
	\]

(H5) if $|z - c_{ij}| > 3k_1\delta$ and $|z - c_{ik}| > 3k_1\delta$, then
 \[
 \ |H^i_{jk} (z)| \le C_{13}\left (\dfrac{\lambda_{ij}^k\alpha_{ij}\delta}{|z - c_{ij}|^2} + \dfrac{\lambda_{ik}^j\alpha_{ik}\delta}{|z - c_{ik}|^2}\right)
 \]
and for all $z \in \mathbb C,$
 \[
 \begin{aligned}
 \ &|h_{J_{il}}(z)| \le C_{14}  L_{J_{il}}(z),  ~ \|h_{J_{il}}\| \le C_{14}, \\
 \  &c_1 (h_{J_{il}}) = 0, ~ c_2 (h_{J_{il}}) = c_2 (g_{J_{il}}).
 \end{aligned}
 \]
\end{lemma}

For $P_0 \subset P(\delta, \alpha_{ij}, F, k_1, C_0),$ define
\[
\ S(P_0) = \bigcup_{J_{il}\in P_0.~j\in J_{il},~ \alpha_{ij} > 0} 2k_1 S_{ij}.
\]

With above lemmas and slight modifications of the proof of (2.35) in \cite{p95}, we have the following lemma (also see \cite[Lemma 6.7]{cy23}). 

\begin{lemma}\label{BasicEstimate2}
Let $\{g_{ij}\}$ be applicable to $P(\delta, \alpha_{ij}, F, k_1, C_0)$ and $\mu \in M_0^+(\C)$.  Let $\{h_{ij}\}$ satisfy the assumption in Lemma \ref{BasicEstimate3}. Set $P_1 = P_1(\delta, \alpha_{ij}, F, k_1, C_0)$ and $P_2 = P_2(\delta, \alpha_{ij}, F, k_1, C_0).$ Let $P_1'\subset P_1$ and $P_2'\subset P_2.$
Suppose that for each complete group $J_{il},$ $h_{J_{il}}$  is constructed as in Lemma \ref{BasicEstimate3}.
 Set $\Psi_{J_{il}}(z) = g_{J_{il}}(z) - h_{J_{il}}(z).$
 Then 
 \begin{eqnarray}\label{FBounded3}
\ \sum_{J_{il}\in P_1'} |\Psi_{J_{il}}(z)| + \sum_{J_{il_i}\in P_2'} |g_{J_{il_i}}(z)| \le C_{15}\min \left (1, \dfrac{\delta}{dist(z, S(P_1'\cup P_2'))} \right )^\frac15. 
\end{eqnarray}
 \end{lemma}
 
\begin{lemma} \label{HIJREst}
Let $\eta_{ij}\in M_0^+(\D(c_{ij}, k_1 \delta) \cap F)$ be such that 
 \[
 \ \|\CT \eta_{ij}\| \le C_0, ~ \|\eta_{ij}\| = \alpha_{ij}.
 \]
Set $h_{ij} = - \CT \eta_{ij}.$ 
Let $N \ge 12k_1$ be an integer. Set $S=S_{i_0,j_0}$ and $J_{j_0}^N = \{j_0-N,...,j_0, ...,j_0+N\}$ for given $i_0$ and $j_0.$ For a given $i_0 - N \le i \le i_0 + N,$ let 
$h_{J_{il}}$ be defined as in Lemma \ref{BasicEstimate3} (H1) for $J_{il} \in P_1.$
Then 
\begin{eqnarray}\label{GammaREstEq}
\  \int_S |h_{J_{il}}(z)| d \area (z) \le C_{17}\area(S) \left (\dfrac{1}{N^2} + \sum_{j\in J_{il}\cap J_{j_0}^N} \dfrac{\alpha_{ij}}{\delta}\right ).
\end{eqnarray}
\end{lemma}

\begin{proof}
Clearly, $h_{ij}$ satisfies \eqref{HConditions}. Let $J_{il}^d,$ $J_{il}^u,$ and $H^i_{jk}$ be defined as in (H1) of Lemma \ref{BasicEstimate3}.
For $ z\in S,$ we have the following estimation cases for $H^i_{jk}.$

Case I: If $j\in J_{il}^d\setminus J_{j_0}^N$ and $k\in J_{il}^u\setminus J_{j_0}^N$, then $|z - c_{ij}| \ge N\delta > 12k_1\delta,$ $|z - c_{ik}| > 12k_1\delta,$ and by (H5) of Lemma \ref{BasicEstimate3}, we have
 \[
 \ |H^i_{jk} (z)| \le C_{18}\dfrac{\lambda_{ij}^k\alpha_{ij} + \lambda_{ik}^j\alpha_{ik}}{N^2\delta}.
 \]
 From (H3) of Lemma \ref{BasicEstimate3}, we infer that
 \[
 \ \begin{aligned}
\ 	H_{il}(z) := & \sum_{j\in J_{il}^d\setminus J_{j_0}^N}\sum_{k\in J_{il}^u\setminus J_{j_0}^N} | H^i_{jk} (z) | \\
\ \le & \dfrac{C_{17}}{N^2\delta} \left ( \sum_{j\in J_{il}^d}\sum_{k\in J_{il}^u}\lambda_{ij}^k\alpha_{ij} + \sum_{j\in J_{il}^d}\sum_{k\in J_{il}^u}\lambda_{ik}^j\alpha_{ik}\right ) \\
\ \le & \dfrac{2C_{18}}{N^2}
\ \end{aligned}
\]

 Case II: If $j\in J_{il}^d\cap J_{j_0}^N$ and $k\in J_{il}^u \setminus J_{j_0}^N$, then $|z - c_{ik}| > 12k_1\delta$ and by (H1) and (H4) of Lemma \ref{BasicEstimate3}, we have 
 \[
 \ |H^i_{jk} (z)| \le C_{19} \left (\dfrac{\lambda_{ik}^j \alpha_{ik}}{N\delta} + \lambda_{ij}^k | \mathcal C(\eta_{ij})(z)| \right ) = C_{19}\lambda_{ij}^k \left (\dfrac{\alpha_{ij}}{N\delta} + |\mathcal C(\eta_{ij})(z)|\right ) ;
 \]

Case III: If $j\in J_{il}^d\setminus J_{j_0}^N$ and $k\in J_{il}^u \cap J_{j_0}^N$, then $|z - c_{ij}| > 12k_1\delta$ and by (H1) and (H4) of Lemma \ref{BasicEstimate3}, we have
 \[
 \ |H^i_{jk} (z)| \le C_{20}\lambda_{ik}^j \left (\dfrac{\alpha_{ik}}{N\delta} + |\mathcal C(\eta_{ik})(z)|\right ) ;
 \]

Case IV: If $j\in J_{il}^d\cap J_{j_0}^N$ and $k\in J_{il}^u \cap J_{j_0}^N$, then by (H1) of Lemma \ref{BasicEstimate3}, we have

 \[
 \ |H^i_{jk} (z)| \le C_{21} (\lambda_{ik}^j |\mathcal C(\eta_{ik})(z)| + \lambda_{ij}^k |\mathcal C(\eta_{ij})(z)|).
 \]
 
Combining Cases I-IV, we estimate $h_{J_{il_d}}$ as the following:
 \[
 \ \begin{aligned}
 \ | h_{J_{il}}(z) | \le & H_{il}(z) + C_{19} \sum_{j\in J_{il}^d\cap J_{j_0}^N} \sum_{k \in J_{il}^u \setminus J_{j_0}^N}  \lambda_{ij}^k \left (\dfrac{\alpha_{ij}}{N\delta} + |\mathcal C(\eta_{ij})(z)|\right ) \\
 \ &  + C_{20} \sum_{j\in J_{il}^d\setminus J_{j_0}^N} \sum_{k \in J_{il}^u \cap J_{j_0}^N}  \lambda_{ik}^j \left (\dfrac{\alpha_{ik}}{N\delta} + |\mathcal C(\eta_{ik})(z)| \right ) \\
 \ & + C_{21} \sum_{j\in J_{il}^d\cap J_{j_0}^N} \sum_{k \in J_{il}^u \cap J_{j_0}^N} (\lambda_{ik}^j |\mathcal C(\eta_{ik})(z)| + \lambda_{ij}^k |\mathcal C(\eta_{ij})(z)|). 
 \end{aligned}
 \]
Using (H2), we get
 \[
 \ \begin{aligned}
 \ | h_{J_{il}}(z) | \le &\dfrac{C_{21}}{N^2} + C_{23}\sum_{j\in J_{il}^d\cap J_{j_0}^N}\left (\dfrac{\alpha_{ij}}{N\delta} + |\mathcal C(\eta_{ij})(z)|\right ) \\
 \ & + C_{23} \sum_{k \in J_{il}^u \cap J_{j_0}^N} \left (\dfrac{\alpha_{ik}}{N\delta} + |\mathcal C(\eta_{ik})(z)| \right ).
 \end{aligned}
 \]
Hence,
 \[
 \ \begin{aligned}
 \ & \int_S |h_{J_{il}}(z)| d \area (z) \\
 \ \le &\dfrac{C_{23}}{N^2} \area(S) + C_{22} \sum_{j\in J_{il}\cap J_{j_0}^N} \left (\dfrac{\alpha_{ij}}{N\delta}\area(S) + \|\eta_{ij}\|(\area(S))^\frac12\right )\\
 \ \le &\dfrac{C_{23}}{N^2} \area(S) + C_{24} \area(S)(\frac 1N + 1)\sum_{j\in J_{il}\cap J_{j_0}^N} \dfrac{\alpha_{ij}}{\delta}.
 \end{aligned}
 \]
 This completes the proof.
\end{proof}

\section{\textbf{The algebra $H^\i (\mathcal D)$}}

For a bounded measurable subset $\mathcal D \subset \C,$ define 
 \[
 \ L^\infty(\area_{\mathcal D}) = \{f\in L^\infty(\mathbb C):~ f(z) = 0,~ z\in \C \setminus \mathcal D\}. 
 \]
 
 \begin{definition} \label{GammaOpenDef}
 The subset $\mathcal D \subset \C$ is \index{$\gamma$-open}{\em $\gamma$-open} if
 there exists a subset $\mathcal Q$ with $\area(\mathcal Q) = 0$ such that for $\lambda \in \mathcal D \setminus \mathcal Q$,
 \begin{eqnarray}\label{DDensityDef}
  \ \lim_{\delta\rightarrow 0} \dfrac{\gamma(\D(\lambda, \delta)\setminus \mathcal D)}{\delta} = 0.
 \end{eqnarray}
 If $\gamma(\mathcal Q) = 0,$ then $\mathcal D$ is called \index{strong $\gamma$-open}{\em strong  $\gamma$-open}.
 \end{definition}
 
 Clearly, by \eqref{AreaGammaEq}, if $\mathcal D$ is strong $\gamma$-open, then $\mathcal D$ is $\gamma$-open. 
 
 \begin{definition}\label{HEDAlgDef}
Let $\mathcal D \subset \C$ be a bounded measurable subset. Assume that $\mathcal D$ is $\gamma$-open. Let $H_{\mathbb C \setminus \mathcal D}$ be the set of functions $f(z)$ such that $f(z)$ is bounded and analytic on $\mathbb C\setminus E_f$ for some compact subset $E_f\subset \mathbb C \setminus \mathcal D.$
Define $H^\i (\mathcal D)$ to be the weak-star closed subalgebra of $L^\i (\area _{\mathcal D})$ generated by functions in $H_{\mathbb C \setminus \mathcal D}.$ 
\end{definition}

If $\mathcal D$ is a bounded open subset, then $H^\i(\mathcal D)$ is the algebra of bounded and analytic functions on $\mathcal D.$ \cite[Main Theorem]{cy22} gives an interesting example that $H^\i(\mathcal D)$ is different from the algebra of bounded and analytic functions on an open subset. 
The aim of this section is to prove the following theorem.

\begin{theorem}\label{HDAlgTheorem}
Let $\mathcal D$ be a $\gamma$-open bounded measurable subset. Let $f\in L^\i (\area_{\mathcal D})$ be given with $\|f\|_{L^\infty (\area_{\mathcal D})}\le 1$. 
If there exists $C_f>0$ (depending on $f$) such that for $\lambda\in \mathbb C$, $\delta > 0$, a smooth non-negative function $\varphi$ with support in $\D(\lambda, \delta),$ $\varphi (z) \le 1,$  and $\left \|\frac{\partial \varphi (z)}{\partial  \bar z} \right \|_\infty \le \frac{C_{25}}{\delta},$ we have   
 \begin{eqnarray}\label{HDAlgTheoremEq}
 \ \left | \int (z-\lambda)^nf(z) \dfrac{\partial \varphi (z)}{\partial  \bar z} d\area_{\mathcal D}(z) \right | \le C_f\delta^n \gamma (\D(\lambda, k_1 \delta) \setminus \mathcal D),
 \end{eqnarray}
 where $k_1\ge 2$ is a given integer and $n \ge 0,$ then there exists a sequence of functions $\{f_n\},$ where $f_n$ is analytic off a compact subset $F_n \subset \C \setminus \mathcal D$ and $\|f_n\|_{\C} \le C_{26}C_f,$ such that $f_n(z) \rightarrow f(z),~\area_{\mathcal D}-a.a..$   
\end{theorem}

We fix $f\in L^\i (\area_{\mathcal D})$ with $\|f\|_{L^\infty (\area_{\mathcal D})}\le 1$ and $f(z) = 0, ~ z \in \C\setminus \mathcal D. $ Replacing $f$ by $\frac{f}{C_f},$ we assume $C_f=1.$
Let $\{\varphi_{ij},S_{ij}, \delta\}$ be a smooth partition of unity as in last section. Set $\alpha_{ij} = \gamma (\D(\lambda, k_1 \delta) \setminus \mathcal D).$ We write $f$ as the following:
\begin{eqnarray}\label{FSum1}
 \ f = \sum_{ij} f_{ij} = \sum_{2S_{ij} \cap \mathcal D \ne \emptyset} f_{ij},
 \end{eqnarray}
 where $f_{ij} = T_{\varphi_{ij}}f.$ From the assumption \eqref{HDAlgTheoremEq}, we see that, for $n \ge 1,$
 \begin{eqnarray}\label{cnfij}
 \ \begin{aligned}
 \ |c_n(f_{ij}, c_{ij})| = &\left | \int (z-\lambda)^{n-1}f(z) \dfrac{\partial \varphi_{ij} (z)}{\partial  \bar z} d\area(z) \right | \\
  \ \le &C_{27}\delta^{n-1} \alpha_{ij}.
 \ \end{aligned}
 \end{eqnarray}

We use the notation $P,P_1,P_2,$ for $F = \C \setminus \mathcal D$ in the last section.

\begin{lemma}\label{GIJHIJDefLemma}
Suppose that $\eta_{ij}\in M_0^+(\D(c_{ij}, k_1 \delta) \setminus \mathcal D)$ satisfying 
\begin{eqnarray} \label{GIJHIJDefLemmaEq1}
\ \|\CT\eta_{ij}\| \le C_0\text{ and }\|\eta_{ij}\| = \alpha_{ij}.
\end{eqnarray}
Then the following statements are true.
\newline
(1) Let $g_{ij}^0 = - \frac{c_1 (f_{ij})}{\alpha_{ij}}\CT \eta_{ij},$ $g_{ij} = f_{ij} - g_{ij}^0,$ and $h_{ij} = - \CT \eta_{ij}.$ Then  $\{g_{ij}\}$ is applicable to $P$ and $\area_{\mathcal D}$ (see Definition \ref{GApplicable}), $\{h_{ij}\}$ satisfies the assumptions in Lemma \ref{BasicEstimate3}, $g_{J_{il}}$ defined in \eqref{GIDefinition} satisfies the properties in Lemma \ref{BasicEstimate}, and $h_{J_{il}}$ defined in \eqref{HFunction} satisfies the properties in Lemma \ref{BasicEstimate3}. 
\newline
(2) 
Rewrite \eqref{FSum1} as the following
 \[
 \ f = \sum_{J_{il}\in P_1} (g_{J_{il}} - h_{J_{il}}) + \sum_{J_{il_i} \in P_2} g_{J_{il_i}} + f_{\delta}
 \]
where
 \begin{eqnarray}\label{FDelta}
 \ f_{\delta} = \sum_{J_{il}\in P} \sum_{j\in J_{il}} g_{ij}^0 + \sum_{J_{il}\in P_1} h_{J_{il}}.
 \end{eqnarray}
 Then $f_{\delta}$ is bounded and analytic off a compact subset of $\C\setminus \mathcal D$ and  
 \begin{eqnarray}\label{FDeltaBDD}
 \ \|f_{\delta}\|_{L^\infty (\area_{\mathcal D})} \le C_{29}.
 \end{eqnarray}
 \newline
 (3) There exists $\delta_n\rightarrow 0$ as $n\rightarrow \i$ such that
 \[
 \ \lim_{n\rightarrow \i}\|f_{\delta_n} - f\|_{L^1 (\area_{\mathcal D})} = 0.
 \]
 \end{lemma}
 
 \begin{proof} (Lemma \ref{GIJHIJDefLemma} (1) and (2)):
 It is clear that
\[
\ |c_n(g_{ij}^0, c_{ij})| \le C_{30} (k_1\delta)^{n-1}\alpha_{ij}, ~ |c_n(h_{ij}, c_{ij})| \le (k_1\delta)^{n-1}\alpha_{ij}.
\]
(1) follows from Lemma \ref{LEProp} and \eqref{cnfij}.

\eqref{FDeltaBDD} in (2) follows from \eqref{FBounded3}.
\end{proof}
 
 To prove Lemma \ref{GIJHIJDefLemma} (3), we need several lemmas. Therefore, for the lemmas below, we assume that there are $\eta_{ij}$ satisfying \eqref{GIJHIJDefLemmaEq1}. The functions $g_{ij}^0,$ $g_{ij},$ $h_{ij},$ $g_{J_{il}},$ $h_{J_{il}},$ and $f_{\delta}$ are defined in Lemma \ref{GIJHIJDefLemma}.
 We fix a positive integer $N$ satisfying $N > 200(2k_1+3)^2$.

\begin{lemma} \label{ACPSetDef}
 There exists $\mathcal Q_f$ with $\area (\mathcal Q_f) = 0$ such that for all $\epsilon > 0,$
 \begin{eqnarray} \label{MContinuous}
\ \lim_{\delta\rightarrow 0} \dfrac{\area (\D(\lambda, \delta))\cap \{ |f(z) - f(\lambda)| > \epsilon\})}{\delta^2} = 0
\end{eqnarray}
  at each point $\lambda\in \mathcal D \setminus \mathcal Q_f.$
 \end{lemma}
 
\begin{proof}
Let $\mathcal D_f \subset \mathcal D$ be the set of Lebesgue points for $f,$ that is, for $\lambda \in \mathcal D_f,$
\[
\ \lim_{\delta\rightarrow 0} \dfrac{1}{\area (\D(\lambda, \delta))} \int _{\D(\lambda, \delta)} |f(z) - f(\lambda)| d \area (z) = 0.
\]
Since
\[
\ \dfrac{\area (\D(\lambda, \delta))\cap \{ |f(z) - f(\lambda)| > \epsilon\})}{\area (\D(\lambda, \delta))} \le\dfrac{1}{\epsilon\area (\D(\lambda, \delta))} \int _{\D(\lambda, \delta)} |f(z) - f(\lambda)| d \area (z),
\]
\eqref{MContinuous} holds at each point of $\mathcal D_f.$ Set $\mathcal Q_f = \mathcal D\setminus \mathcal D_f,$ then $\area (\mathcal Q_f) = 0$ by 
the Lebesgue differentiation theorem.
\end{proof}

Let $\lambda\in \mathcal D$ satisfy \eqref{DDensityDef} 
and be such that $f$ satisfies \eqref{MContinuous} at $\lambda.$ Define
 \begin{eqnarray}\label{FDef}
 \ F(f, N, \lambda) = \left \{z: ~ |f(z) - f(\lambda )| \ge \frac{1}{N^3} \right \}. 
 \end{eqnarray}
 Let $\mathcal D(f, m, N)$ be the set of $\lambda \in \mathcal D$ such that
 \begin{eqnarray}\label{DFMNDef1}
 \ \area (F(f, N, \lambda)\cap \D(\lambda, \delta_1)) < \dfrac{1}{N^8} \delta_1^2
 \end{eqnarray}
and
 \begin{eqnarray}\label{DFMNDef2}
 \ \gamma ( \D(\lambda, \delta_1) \setminus \mathcal D) < \dfrac{1}{N^2} \delta_1
 \end{eqnarray}
for $\delta_1 \le \frac{1}{m}$. From Definition \ref{GammaOpenDef} and Lemma \ref{ACPSetDef}, it is straightforward to verify
 \begin{eqnarray}\label{denseSet}
 \ \area\left (\mathcal D \setminus \bigcup_{m=1}^\infty \mathcal D(f, m, N) \right ) = 0.
 \end{eqnarray}
 
 Set $S=S_{i_0,j_0}$ for given $i_0,j_0.$ and $J_{k}^N = \{k-N,...,k, ...,k+N\}.$
 
 \begin{lemma} \label{FIJREst}
If $\delta < \frac{1}{2mN}$ and $\mathcal D(f, m, N)\cap S \ne \emptyset,$ then 
\begin{eqnarray}\label{GammaREstEq}
\  \sum_{i\in J_{i_0}^N, j\in J_{j_0}^N} \int_{\mathcal D(f, m, N)\cap S} |f_{ij}(z)| d \area (z) \le \dfrac{C_{31}}{N} \area(S).
\end{eqnarray}
\end{lemma}

\begin{proof}
Let $\lambda\in \mathcal D(f, m, N)\cap S.$
For $z, w \in F(f, N, \lambda)^c$, by \eqref{FDef}, we have $|f(z) - f(w)| < \frac{2}{N^3}$ and for $i\in J_{i_0}^N, j\in J_{j_0}^N,$
 \[
 \ \begin{aligned}
 \ |f_{ij}(w )| \le &\dfrac{1}{\pi} \int_{F(f, N, \lambda)}\dfrac{2\|f\|}{|z-w|}|\bar\partial \varphi _{ij}|d\area(z) \\
 \ &+ \dfrac{1}{N^3\pi} \int_{F(f, N, \lambda)^c}\dfrac{2}{|z-w|}|\bar\partial \varphi _{ij}|d\area(z) \\
 \ \le & C_{32}\sqrt{\area(\D(\lambda, 2N\delta) \cap F(f, N, \lambda))} \|\bar\partial \varphi _{ij}\| + C_{32}\dfrac{1}{N^3} \\
 \ \le &\dfrac{C_{33}}{N^3},
 \end{aligned}
 \]
 where \eqref{DFMNDef1} is used for the last step.
 Therefore,
 \begin{eqnarray}\label{IEstimate}
 \ \begin{aligned}
 \ & \sum_{i\in J_{i_0}^N, j\in J_{j_0}^N}\int_{S \cap \mathcal D(f, m, N)} |f_{ij}(w)| d\area(w) \\
\ \le & C_{34}\sum_{i\in J_{i_0}^N, j\in J_{j_0}^N} \left ( \dfrac{\area(S)}{N^3} + \int_{S \cap F(f, N, \lambda)} |f_{ij}(w)| d\area(w) \right ) \\
  \le &C_{35} (2N+1)^2\left ( \dfrac{\area(S)}{N^3} + \area(\D(\lambda, 2\delta) \cap F(f, N, \lambda)) \right ) \\
 \ \le &\dfrac{C_{36}}{N}\area(S),
 \ \end{aligned}
 \end{eqnarray}   
\end{proof}
where \eqref{DFMNDef1} is used again for the last step.

\begin{lemma} \label{GammaREst}
If $\delta < \frac{1}{2mN}$ and  $\lambda\in \mathcal D(f, m, N)\cap S \ne \emptyset,$ then
\begin{eqnarray}\label{GammaREstEq}
\ \sum_{i\in J_{i_0}^N, j\in J_{j_0}^N}\alpha_{ij} \le 100(2k_1+3)^2 \gamma (\D(\lambda, 2N\delta) \setminus \mathcal D) \le 200(2k_1+3)^2\dfrac{\delta}{N}.
\end{eqnarray}
\end{lemma}

\begin{proof}
Let $F = \D(\lambda, 2N\delta) \setminus \mathcal D.$ Then, by \eqref{DFMNDef2}, $\gamma(F) \le \frac{2N\delta}{N^2} = \frac{2\delta}{N}.$ Therefore, for each $\lambda \in \C,$  the disk $\D(\lambda, \gamma(F))$ meets at most $(2k_1+3)^2$ of sets $\{\D(c_{ij}, k_1\delta) \setminus \mathcal D\}_{i\in J_{i_0}^N, j\in J_{j_0}^N}.$  Now \eqref{GammaREstEq} follows from \cite[Theorem VIII.2.7]{gamelin}. 
\end{proof}

\begin{lemma} \label{GIJREst}
If $\delta < \frac{1}{2mN}$ and  $\mathcal D(f, m, N)\cap S \ne \emptyset,$ then
\begin{eqnarray}\label{GIJREstEq}
\  \sum_{i\in J_{i_0}^N, j\in J_{j_0}^N} \int_S |g_{ij}^0(z)| d \area (z) \le \dfrac{C_{37}}{N} \area(S).
\end{eqnarray}
\end{lemma}

\begin{proof}
We have the following calculation.
\[
 \ \begin{aligned}
 \ \int_S |g^0_{ij}(w)| d\area(w) \le & C_{38}  \int \int_S \dfrac{1}{|z-w | }d\area (w) d\eta_{ij}(z) \\
 \ \le & C_{39} \sqrt{\area(S)} \alpha_{ij}.
 \ \end{aligned}
 \]
 The proof now follows from Lemma \ref{GammaREst}.
\end{proof}

\begin{lemma} \label{KeySEstimate} 
If $\delta < \frac{1}{2mN}$ and   $S \cap \mathcal D(f, m, N) \ne \emptyset,$ then
\begin{eqnarray}\label{keyEstimate1}
 \ \int _{S \cap \mathcal D(f, m, N)} |f(z) - f_{\delta}(z)| d\area(z) \le \dfrac{C_{40}}{N^\frac15} \area(S).
 \end{eqnarray}
\end{lemma} 

\begin{proof}
Set $\Psi_{J_{il}} = g_{J_{il}} - h_{J_{il}}$ Define
\[
\ R_1 = \{J_{il}\in P_1:~ i\in J_{i_0}^N\text{ and } J_{il}\cap J_{j_0}^N \ne \emptyset \}
\]
and
\[
\ R_2 = \{J_{il_i}\in P_2:~ i\in J_{i_0}^N\text{ and } J_{il_i}\cap J_{j_0}^N \ne \emptyset \}.
\]
By definition of a complete group (see \eqref{IndexDefinition1}, \eqref{IndexDefinition2}, and \eqref{IndexDefinition3}) and Lemma \ref{GammaREst}, for each $i\in J_{i_0}^N,$ there are at most two $J_{il}\in R_1.$ So we rewrite
\[
\ R_1 = \{J_i^d,J_i^u:~i\in J_{i_0}^N\},
\]
where $j_1 < j_2$ for $j_1\in J_i^d$ and $j_2\in J_i^u.$ 
For $z\in S,$ from Lemma \ref{BasicEstimate2}, we see that
\begin{eqnarray}\label{KeySEstimateEq1}
\ \sum_{J_{il}\in P_1 \setminus R_1} |\Psi_{J_{il}}(z)| + \sum_{J_{il_i}\in P_2 \setminus R_2} |g_{J_{il_i}}(z)| \le \dfrac{C_{41}}{N^\frac 15}. 
\end{eqnarray}
For $i\in J_{i_0}^N,$ $J_i^d,J_i^u\in R_1,$ and $J_{il_i}\in R_2,$ The index groups $J_i^d\setminus J_{j_0}^N,$ $J_i^u\setminus J_{j_0}^N,$ and $J_{il_i}\setminus J_{j_0}^N$ can be viewed as incomplete groups. Hence, using Lemma \ref{BasicEstimate2}, we get
\[
\ \sum_{i\in J_{i_0}^N} (|g_{J_i^d\setminus J_{j_0}^N}(z)| + |g_{J_i^u\setminus J_{j_0}^N}(z)| + |g_{J_{il_i}\setminus J_{j_0}^N}(z)|) \le \dfrac{C_{42}}{N^\frac 15} 
\]
Combining with Lemma \ref{FIJREst} and Lemma \ref{GIJREst}, we have
\begin{eqnarray}\label{KeySEstimateEq2}
\ \int_{S \cap \mathcal D(f, m, N)} \left (\sum_{J_{il}\in R_1} |g_{J_{il}}(z)| + \sum_{J_{il_i}\in R_2} |g_{J_{il_i}}(z)|\right )d\area(z) \le \dfrac{C_{43}}{N^\frac 15}\area(S). 
\end{eqnarray}
Using Lemma \ref{HIJREst} and Lemma \ref{GammaREst}, we conclude that
\begin{eqnarray}\label{KeySEstimateEq3}
\ \int_{S \cap \mathcal D(f, m, N)} \sum_{i\in J_{i_0}^N} (|h_{J_i^d}(z)| + |h_{J_i^u}(z)|) d\area(z) \le \dfrac{C_{44}}{N}\area(S). 
\end{eqnarray}
Combining \eqref{KeySEstimateEq1}, \eqref{KeySEstimateEq2}, and \eqref{KeySEstimateEq3}, we prove the lemma.
\end{proof}

Since $\area (\mathcal D) < \i,$ we obtain the following corollary.

\begin{corollary}\label{keyEstimate}
If $\delta < \frac{1}{2mN},$ then
 \[
 \ \int _{\mathcal D(f, m, N)} |f(z) - f_{\delta}(z)| d\area(z) \le \dfrac{C_{45}}{N^\frac15}.
 \]	
\end{corollary}

\begin{proof}
(Lemma \ref{GIJHIJDefLemma} (3)):
There exists $m_N$ such that, from \eqref{denseSet}, 
 \[
 \ \area(\mathcal D \setminus \mathcal D(f, m_N, N) ) < \dfrac{1}{N}.
 \]
Set $\delta_N = \frac{1}{4m_NN}$. Clearly, by Corollary \ref{keyEstimate} and \eqref{FDeltaBDD},
 \[
 \ \|f_{\delta_N} - f\|_{L^1 (\area_{\mathcal D})}\rightarrow 0.
 \]
\end{proof}

\begin{proof} 
(Theorem \ref{HDAlgTheorem}): 
Applying Theorem \ref{TolsaTheorem} (1), there exists $\eta_{ij}^0\in M_0^+(\D(c_{ij}, k_1 \delta) \setminus \mathcal D)$ such that $\|\eta_{ij}^0\| \ge \frac{1}{2A_T}\alpha_{ij}$ and $\|\CT(\eta_{ij}^0)\|_{\C} \le 1.$ Set $\eta_{ij} = \frac{\eta_{ij}^0}{\|\eta_{ij}^0)\|}\alpha_{ij}.$ Then $\eta_{ij}$ satisfies \eqref{GIJHIJDefLemmaEq1}.

From Lemma \ref{GIJHIJDefLemma} (3), we see that $f_N := f_{\delta_N} \rightarrow f,~ \area_{\mathcal D}-a.a.$ (by passing to a subsequence). Thus, $f_N\rightarrow f$ in $L^\infty(\area_{\mathcal D})$ weak-star topology. This implies $f\in H^\infty(\mathcal D)$ as $f_N\in H^\infty(\mathcal D)$. 
 This proves the theorem.
\end{proof}

\section{\textbf{Non-removable boundary and removable set}}

For a compact subset $K\subset \C$ and $\mu\in M_0^+(K),$ we assume $\rikmu$ is pure.

\begin{definition}\label{EDef}
The envelope $E$ for $K$ and $\mu$ is the set of points $\lambda \in  K$ such that there exists $\mu_\lambda\in M_0(K)$ that is absolutely continuous with respect to $\mu$ such that $\mu_\lambda (\{\lambda\}) = 0$ and $\int f(z) d\mu_\lambda(z) = f(\lambda)$ for each $f \in R(K).$ 
\end{definition}

The elementary properties of $E$ are listed below.

\begin{proposition}\label{EProp} (a) $E$ is the set of weak-star continuous homomorphisms on $\rikmu$ (see \cite[Proposition VI.2.5]{conway}).
\newline
(b) $E$ is a nonempty Borel set with area density one at each of its points (see \cite[Proposition VI.2.8]{conway}).
\newline
(c) $\text{int}(E) = \text{int}(\overline E)$ (see \cite[Proposition VI.3.9]{conway}).
\newline
(d) $\overline E$ is the union of $\text{spt}\mu$ and some collection of bounded components $\C \setminus \text{spt}\mu$ (see \cite[Proposition VI.3.11]{conway}).
\end{proposition}

For $\lambda\in E$ and $f\in R^\i (K,,\mu),$ set $\rho(f)(\lambda) = \int fd\mu_\lambda.$ Clearly $\rho(f)(\lambda)$ is independent of the particular $\mu_\lambda$ chosen. We thus have a map $\rho$ called Chaumat's map for $K$ and $\mu,$ which associates
to each function in $\rikmu$ a point function on $E.$ 
Chaumat's Theorem \cite{cha74} states as the following: The map $\rho$ is an isometric isomorphism and a weak-star homeomorphism from $\rikmu$ onto $R^\i (\overline E, \area_E)$ (also see \cite[Chaumat's Theorem on page 288]{conway}). Therefore, we make the following assumptions:
\newline
(a) $\rikmu$ is pure (has no non-trivial $L^\i$ summands); 
\newline
(b) $K = \overline E,$ where $E$ is the envelope for $R^\i (K, \mu);$
\newline
(c) $\mu =\area_E.$

\begin{definition}\label{R0ForRIKMU}
Let $\Lambda = \{g_n\} \subset R(\overline E)^\perp \cap L^1(\area_E)$ be a $L^1(\area_E)$ norm dense subset.
 The non-removable boundary for $\Lambda$ is defined by
 \[
 \ \mathcal F (\Lambda) = \bigcap_{n = 1}^\infty \{z:~\lim_{\epsilon \rightarrow 0} \mathcal C_\epsilon(g_n\area_E)(z)\text{ exists, } \mathcal C(g_n\area_E)(z) = 0\}
 \]
 and the removable set for $\Lambda$ is defined by
 \[
 \ \mathcal R (\Lambda) = \bigcup_{n = 1}^\infty \{z:~\lim_{\epsilon \rightarrow 0} \mathcal C_\epsilon(g_n\area_E)(z)\text{ exists, } \mathcal C(g_n\area_E)(z) \ne 0\}.
 \]
 \end{definition}
 
 Clearly,
 \begin{eqnarray}\label{RFProp1}
\ \mathcal R (\Lambda)  \cap \mathcal F (\Lambda)  = \emptyset\text{ and }\mathcal R (\Lambda)  \cup \mathcal F (\Lambda)  \approx \C,~ \gamma-a.a..	
\end{eqnarray}

 The concept of non-removable boundary and removable set was first introduced by \cite{cy22} for a string of beads set. Conway and Yang \cite{cy23} extended the concept to an arbitrary compact subset $K\subset \C$ and $\mu\in M_0^+(K).$ We will see that  $\mathcal R (\Lambda)$ and $\mathcal F (\Lambda) $ are more appropriate than that of the
envelope $E$ in studying $\rikmu.$

\begin{lemma} \label{CTZeroOnF}
If $g\in R(\overline E)^\perp \cap L^1(\area_E),$ then
\[
\ \mathcal C(g\area_E)(\lambda) = 0,~ \gamma|_{\mathcal F(\Lambda)}-a.a..  
\]	
\end{lemma}

\begin{proof}
Assume there exists a compact subset $F_1\subset \mathcal F (\Lambda)$ such that $\gamma (F_1) > 0$ and  
 \begin{eqnarray}\label{FCForRAssump}
 \ \mathcal C(g_n\area_E)(\lambda) = 0 \text{ for }n \ge 1\text{ and } Re(\mathcal C(g\area_E))(\lambda) > 0,~ \lambda\in F_1.
 \end{eqnarray}
 Applying Lemma \ref{BBFunctLemma} for $\{g\} \cup \Lambda \subset L^1(\area_E)$ and $F_1,$ we get $\eta\in M_0^+(F_1)$ and $f(z)$ satisfy (1)-(3) of Lemma \ref{BBFunctLemma}. By \eqref{BBFunctLemmaEq1}, for $n\ge 1,$
 \[
  \ \int \mathcal C(g_n\area_E)(z) d\eta(z) = - \int f(z) g_n(z) d \area_E(z)= 0,
  \]
  which implies $f\in R^\i(\overline E, \area_E)$ by the Hahn-Banach Theorem. Applying \eqref{BBFunctLemmaEq1} to $g,$ we get
  \[
  \ \int Re(\mathcal C(g\area_E)(z)) d\eta(z) = - Re \left ( \int f(z) g(z) d \area_E(z)\right ) = 0,
  \]
  which implies $Re(\mathcal C(g\mu)(z)) = 0,~ \eta-a.a..$ This contradicts  \eqref{FCForRAssump}.	
\end{proof}

The following corollary follows from Lemma \ref{CTZeroOnF}.

\begin{corollary}\label{RFUnique}
Let $\Lambda = \{g_n\},~ \Lambda' = \{g_n'\} \subset R(\overline E)^\perp \cap L^1(\area_E)$ be two dense subsets. Then
\[
\ \ \mathcal F (\Lambda) \approx \ \mathcal F (\Lambda')\text{ and }\mathcal R (\Lambda) \approx \ \mathcal R (\Lambda'),~\gamma-a.a..
\]
\end{corollary}

Hence, $\mathcal F(\Lambda)$ and $\mathcal R(\Lambda)$ are independent of choices of $\Lambda$ up to a set of zero analytic capacity. We will simply use $\mathcal F$ for $\mathcal F(\Lambda)$ and $\mathcal R$ for $\mathcal R(\Lambda).$

\begin{lemma} \label{CTGammaCE}
Let $g\in L^1(\area_E).$ Then there exists $\mathcal Q$ with $\gamma (\mathcal Q) = 0$ such that for $\lambda \in \C\setminus \mathcal Q,$ $\Theta_{g\area_E}(\lambda) = 0$ and $\CT(g\area_E)$ is $\gamma$-continuous at $\lambda.$
\end{lemma}

\begin{proof}
By Lemma \ref{RNDecom2} and Corollary \ref{ZeroAC}, there exists $\mathcal Q$ with $\gamma (\mathcal Q) = 0$ such that for $\lambda \in \C \setminus \mathcal Q,$ $\Theta_{g\area_E}(\lambda) = 0$ and the principle values of $\CT(g\area_E)(\lambda)$ exists.
From Lemma \ref{CauchyTLemma}, we see that $\CT(g\area_E)(z)$ is $\gamma$-continuous at $\lambda.$
\end{proof} 

The following corollary for $\mathcal R$ is a generalization of Proposition \ref{EProp} (b) for $E.$

\begin{corollary}\label{RSGOpen}
The removable set $\mathcal R$ is strong $\gamma$-open.	
\end{corollary}

\begin{proof}
From Lemma \ref{CTGammaCE}, for $\lambda \in \mathcal R,~ \gamma-a.a.,$ there exists $g_n$ such that $\Theta_{g_n\area_E} (\lambda) = 0,$ $\CT (g_n\area_E)(\lambda) \ne 0,$ and $\CT (g_n\area_E)(z)$ is $\gamma$-continuous at $\lambda.$ Therefore,
\[
\begin{aligned}
\ &\lim_{\delta\rightarrow 0} \dfrac{\gamma (\D(\lambda,\delta)\setminus \mathcal R)}{\delta} \\
\ \le &\lim_{\delta\rightarrow 0} \dfrac{\gamma\left (\D(\lambda,\delta)\cap \left \{|\CT (g_n\area_E)(z) - \CT (g_n\area_E)(\lambda)| > \frac{|\CT (g_n\area_E)(\lambda)|}{2}\right \} \right )}{\delta} \\
\ = & 0.
\end{aligned}
\]
So $\mathcal R$ is strong $\gamma$-open.	
\end{proof}

Recall $ \tilde \nu$ is defined by \eqref{CTAIDefinition} for $\nu \in M_0(\C)$.	
For $g \in R(\overline E)^\perp \cap L^1(\area_E),$ define
\begin{eqnarray} \label{R0GDef}
\  E_0(g\area_E) = \{\lambda:~ \widetilde {g\area_E}(\lambda) < \infty,~ \mathcal C(g\area_E)(\lambda) \ne 0 \}.	
\end{eqnarray} 
Clearly, $ E_0(g\area_E) \subset E$ because for $\lambda\in E_0(g\area_E),$ $\mu_\lambda = \frac{g\area_E(z)}{\CT (g\area_E)(\lambda) (z-\lambda)}$ satisfies Definition \ref{EDef}.

\begin{proposition}\label{EMeasurable}
There exists a sequence $\{g_n\} \subset R(\overline E)^\perp \cap L^1(\area_E)$ and there ecists a
doubly indexed  sequence of open balls $\{\D (\lambda_{ij}, \delta_{ij})\}$ such that
\[
\ E \approx  \bigcup_{n=1}^\infty E_0(g_n\area_E) \approx \bigcap_{j=1}^\infty \bigcup_{i=1}^\infty \D (\lambda_{ij}, \delta_{ij}), ~ \area-a.a.
\] 
\end{proposition}

\begin{proof} 
	For $\lambda \in E$, let $\mu_\lambda$ be as in Definition \ref{EDef}. Let $g_\lambda(z) =  (z - \lambda)\frac{d\mu_\lambda}{d\area_E},$ then $g_{\lambda}\in R(\overline E)^\perp \cap L^1(\area_E)$ such that $\widetilde {g_{\lambda}\area_E}(\lambda) < \infty$ and $\mathcal C(g_{\lambda}\area_E)(\lambda) \ne 0$. By Lemma \ref{CauchyTLemma} and \eqref{AreaGammaEq}, we have
 \begin{eqnarray}\label{BorelSetEq1}
 \begin{aligned}
\ & \lim_{\delta \rightarrow 0} \dfrac{\area (\D(\lambda, \delta) \setminus  E_0(g_{\lambda}\area_E))}{\delta ^2} \\
\ \le & 4 \pi \lim_{\delta \rightarrow 0} \dfrac{\left (\gamma \left (\D(\lambda, \delta) \cap \left \{|\mathcal C(g_{\lambda}\area_E)(z) - \mathcal C(g_{\lambda}\area_E)(\lambda)| > \frac{\mathcal C(g_{\lambda}\area_E)(\lambda)|}{2} \right \}\right  )\right )^2}{\delta ^2} \\
\  = & 0.
\end{aligned}
\end{eqnarray}
There is $\delta_{\lambda}^n > 0$ such that from \eqref{BorelSetEq1}, we have
 \[
 \ \area (\D(\lambda, \delta) \setminus  E_0(g_{\lambda}\area_E)) \le \dfrac{1}{n}\delta ^2, ~ \delta < \delta_{\lambda}^n. 
 \]
By $5r$-covering theorem (Theorem 2.2 in \cite{Tol14}), we can find a sequence of disjoint disks $\{\D (\lambda_{mn}, \frac 15\delta_{\lambda_{mn}}^n)\}$ such that 
 \[
 \ \bigcup_{m=1}^\infty E_0(g_{\lambda_{mn} }\area_E) \subset  E \subset \bigcup_{\lambda \in E} \D(\lambda, \frac 15 \delta_{\lambda}^n) \subset \bigcup_{m=1}^\infty 5 \D(\lambda_{mn}, \frac 15\delta_{\lambda_{mn}}^n).   
 \]
Hence,
 \[
 \ \bigcup_{n=1}^\infty\bigcup_{m=1}^\infty E_0(g_{\lambda_{mn} }\area_E) \subset  E \subset \bigcap_{n=1}^\infty\bigcup_{m=1}^\infty 5 \D(\lambda_{mn}, \frac 15\delta_{\lambda_{mn}}^n)   
 \]
and
 \[
 \begin{aligned}
 \ & \area \left (\bigcup_{m=1}^\infty \D(\lambda_{mn}, \delta_{\lambda_{mn}}^n) \setminus \bigcup_{m=1}^\infty E_0(g_{\lambda_{mn} }\area_E) \right ) \\
 \ \le & \sum_{m=1}^\infty  \area \left (\D(\lambda_{mn}, \delta_{\lambda_{mn}}^n) \setminus  E_0(g_{\lambda_{mn} }\area_E) \right ) \\
 \ \le & \dfrac{1}{n}\sum_{m=1}^\infty (\delta_{\lambda_{mn}}^n)^2.   
 \end{aligned}
 \]
Notice that $\sum_{m=1}^\infty (\delta_{\lambda_{mn}}^n)^2$ is bounded, therefore,
 \[
 \ E \approx \bigcap_{n=1}^\infty \bigcup_{m=1}^\infty \D(\lambda_{mn}, \delta_{\lambda_{mn}}^n), ~ \area-a.a.
 \]
since 
 \[
 \ \area \left (\bigcap_{n=1}^\infty \bigcup_{m=1}^\infty \D(\lambda_{mn}, \delta_{\lambda_{mn}}^n) \setminus \bigcup_{n=1}^\infty  \bigcup_{m=1}^\infty E_0(g_{\lambda_{mn} }\mu) \right ) = 0.
 \]
\end{proof}

\begin{corollary}\label{EREqual}
$E\approx \mathcal R,~\area-a.a.$ and $\overline E = \overline {\mathcal R}.$	
\end{corollary}

\begin{proof}
Let $\Lambda = \{g_n\}$ be as in Proposition \ref{EMeasurable}. We may assume that $\Lambda$ is dense in $R(\overline E)^\perp \cap L^1(\area_E).$
If $\lambda \in E_0(g_n\area_E),$ then $\lim_{\epsilon \rightarrow 0} \CT(g_n\area_E)(\lambda)$ exists since $\widetilde{g_n\area_E}(\lambda) < \i.$ Hence, $E_0(g_n\area_E) \subset \mathcal R.$ 
The inclusion $\mathcal R \subset \cup_{n=1}^\i E_0(g_n\area_E), ~ \area-a.a..$ follows from Fubini's theorem.

The inclusion $\overline E \subset \overline {\mathcal R}$ follows from Lemma \ref{CTZeroOnF}, while $\overline {\mathcal R} \subset \overline E$ follows from Corollary \ref{RSGOpen}. 
\end{proof}

\section{\textbf{A criterion for functions in $R^\i(\overline E,\area_E)$}}

In this section, we prove the following criterion.

\begin{theorem}\label{REEAlgTheorem}
Let $f\in L^\i (\area_E)$ be given with $\|f\|_{L^\infty (\area_E)}\le 1$. 
Then the following statements are equivalent.

(1) $f \in R^\i(\overline E,\area_E).$
  
 (2) There exists a constant $C_f>0$ (only depending on $f$) and an integer $k_1\ge 2$ so that for all $\lambda\in \mathbb C$, $\delta > 0$, and for all choices of a smooth non-negative function $\varphi$ with support in $\D(\lambda, \delta)$ satisfying $\varphi \le 1$ and $\left \|\frac{\partial \varphi (z)}{\partial  \bar z} \right \|_\infty \le \frac{C_{45}}{\delta},$ we have 
 \[
 \ \left | \int (z-\lambda)^nf(z) \dfrac{\partial \varphi (z)}{\partial  \bar z} d\area_E(z) \right | \le C_f\delta^n \gamma (\D(\lambda, k_1 \delta) \setminus \mathcal R)
 \] 
 for all integers $n \ge 0.$
 
 (3) $f \in H^\i(\mathcal R).$
 \end{theorem}

\begin{proof}

(1) $\Rightarrow$ (2): For $f \in R^\i (\overline E, \area_E),$ we find $\{f_n\}\subset R(\overline E)$ such that $\|f_n\| \le C_f$ and $f_n(z)\rightarrow f(z),~\area_E-a.a.$ (see \cite[Proposition VI.22.3]{conway}). 
 Using \cite[Theorem VIII.8.1]{gamelin}, for a smooth $\varphi$ with support in $\D(\lambda, \delta),$ we get
\[
 \ \begin{aligned}
 \ & \left | \int f_n(z) \dfrac{\partial \varphi (z)}{\partial  \bar z} d\area_E(z) \right | \\
 \ \le & \left | \int _{\overline E} f_n(z) \dfrac{\partial \varphi (z)}{\partial  \bar z} d\area(z) \right | + \left | \int  _{\overline E \setminus E} f_n(z) \dfrac{\partial \varphi (z)}{\partial  \bar z} d\area_E(z) \right | \\
 \ \le & C_{46}C_f\delta \left \|\dfrac{\partial \varphi (z)}{\partial  \bar z} \right \|_\infty \left (\gamma (\D(\lambda, k_1\delta) \setminus \overline E) + \sqrt{\area (\D(\lambda, k_1\delta) \cap (\overline E \setminus E))}\right ).
 \end{aligned}
 \]
 From Corollary \ref{EREqual}, we get $\area_{\mathcal R} = \area_E.$ Using \eqref{AreaGammaEq}, we have
\[
\ \sqrt{\area (\D(c_{ij}, k_1\delta) \cap (\overline E \setminus E))} \le \sqrt{\area (\D(c_{ij}, k_1\delta) \setminus \mathcal R)} \le \sqrt{4 \pi} \gamma (\D(c_{ij}, k_1\delta) \setminus \mathcal R).
\]
Hence, since $\mathcal R \subset \overline E,$
\[
\ \left | \int f_n(z) \dfrac{\partial \varphi (z)}{\partial  \bar z} d\area_E(z) \right | \le C_{46}C_f(1 + \sqrt{4\pi})\delta \left \|\dfrac{\partial \varphi (z)}{\partial  \bar z} \right \|_\infty \gamma (\D(c_{ij}, k_1\delta) \setminus \mathcal R).
\]
Taking weak-star limits for $n\rightarrow \i$ and replacing $\varphi$ by $(z-\lambda)^n\varphi,$ we prove (2).

(2) $\Rightarrow$ (3) follows from Theorem \ref{HDAlgTheorem} since $\mathcal R$ is $\gamma$-open from Corollary \ref{RSGOpen}.

(3) $\Rightarrow$ (1): Let $f$ be a bounded and analytic function outside of a compact subset $E_f\subset \C\setminus \mathcal R$ with $\|f\| \le 1.$ 
For $\lambda\in \mathbb C$, $\delta > 0$, and a smooth function $\varphi$ with support in $\D(\lambda, \delta)$, we see that $\frac{T_{\varphi}f}{\|T_{\varphi}f\|}$ is analytic off  $\text{supp}(\varphi) \cap E_f.$ Using \eqref{TOperProp}, we have the following calculation.
\[
\ \begin{aligned}
\ |(T_{\varphi}f)'(\i)| = & \dfrac{1}{\pi} \left | \int f(z) \dfrac{\partial \varphi (z)}{\partial  \bar z} d\area(z) \right | \\
\ \le & \|T_{\varphi}f\|\gamma(\text{supp}(\varphi) \cap E_f) \\ 
\ \le & C_{47}\delta \left \|\dfrac{\partial \varphi (z)}{\partial  \bar z} \right \|_\infty \gamma (\D(\lambda, \delta) \setminus \mathcal R).
\ \end{aligned}
\]
Since, by \eqref{AreaGammaEq},
\[
\ \begin{aligned}
\ \left | \int _{\C \setminus \mathcal R} f(z) \dfrac{\partial \varphi (z)}{\partial  \bar z} d\area(z) \right | \le & C_{48}\left \|\dfrac{\partial \varphi (z)}{\partial  \bar z} \right \|_\infty \area (\D(\lambda, \delta) \setminus \mathcal R) \\
\  \le & C_{49}\delta \left \|\dfrac{\partial \varphi (z)}{\partial  \bar z} \right \|_\infty \gamma (\D(\lambda, \delta) \setminus \mathcal R), 
\ \end{aligned}
\]
we get
\[
\ \left | \int f(z) \dfrac{\partial \varphi (z)}{\partial  \bar z} d\area_E(z) \right | \le C_{50}\delta \left \|\dfrac{\partial \varphi (z)}{\partial  \bar z} \right \|_\infty \gamma (\D(\lambda, \delta) \setminus \mathcal R). 
\]
Therefore, the assumption of Theorem \ref{HDAlgTheorem} holds 
if we replace $\varphi$ by $(z-\lambda)^n\varphi$ above. 
We now use the notation in the proof of Theorem \ref{HDAlgTheorem} for $\mathcal D = \mathcal R.$ Applying Lemma \ref{BBFunctLemma},  we get $\eta_{ij}^0\in M_0^+(\D(c_{ij}, k_1\delta) \setminus \mathcal R)$ for $\Lambda = \{g_n\},$ a dense subset of $R(\overline E)^\perp \cap L^1(\area_E)$, such that $\|\eta_{ij}^0\| \ge \frac{1}{2C_5}\alpha_{ij},$ where $\alpha_{ij} = \gamma(\D(c_{ij}, k_1\delta) \setminus \mathcal R).$ By \eqref{BBFunctLemmaEq1}, $\int \CT(\eta_{ij}^0)g_n d\area_E = 0$ for $n \ge 1$ since $\CT(g_n\area_E)(z) = 0,~\eta_{ij}^0-a.a.$ by Lemma \ref{ZeroACEta}. Hence, using the Hahn-Banach Theorem, we see that $\CT(\eta_{ij}^0) \in R^\i(\overline{E}, \area_E).$ Set $\eta_{ij} = \frac{\eta_{ij}^0}{\|\eta_{ij}^0\|}\alpha_{ij}.$ Then $\eta_{ij}$ satisfies \eqref{GIJHIJDefLemmaEq1} and $f_\delta \in R^\i(\overline E, \area_E),$ where $f_\delta$ is defined as in \eqref{FDelta}. Applying Lemma \ref{GIJHIJDefLemma} (3), we see that $f \in R^\i(\overline E, \area_E).$ Thus, $H^\i(\mathcal R) \subset R^\i(\overline E, \area_E)$ since those functions $f$ above are weak-star dense in $H^\i(\mathcal R).$  	
\end{proof}

\begin{corollary}\label{CTInRI}
Let $\eta \in M_0^+(\C)$ and $\|\CT\eta\|_{L^\i(\C)} \le 1.$ Suppose that $\CT\eta \in R^\i(\overline E, \area_E).$ If $w\in L^\i(\eta)$ such that $\|w\|_{L^\i(\eta)} \le 1$ and $\|\CT(w\eta)\|_ {L^\i(\C)} \le 1,$ then $\CT(w\eta) \in R^\i(\overline E, \area_E).$  
\end{corollary}

\begin{proof}
For a non-negative smooth function $\phi$ supported in $\D(\lambda, 2\delta)$ with $0\le \phi \le 1,$ $\phi(z) 
= 1$ for $z \in \D(\lambda, \delta),$ and $\|\bar \partial \phi\| \le \frac{C_{51}}{\delta},$ we infer that, from Theorem \ref{REEAlgTheorem} (2), \eqref{CTDistributionEq}, and \eqref{AreaGammaEq},
\[
 \begin{aligned}
 \ \eta(\D(\lambda, \delta)) \le & C_{52} \left | \int \CT \eta (z) \dfrac{\partial \phi (z)}{\partial  \bar z} d\area(z) \right | \\
 \le & C_{52} \left | \int \CT \eta (z) \dfrac{\partial \phi (z)}{\partial  \bar z} d\area_E(z) \right | + C_{52} \left \| \dfrac{\partial \phi (z)}{\partial  \bar z}\right \| \area (\D(\lambda, 2\delta) \setminus \mathcal R)  \\
 \ \le & C_\eta \gamma (\D(\lambda, 2k_1\delta) \setminus \mathcal R),
 \end{aligned}
 \]
 since $\CT \eta \in R^\i (\overline E, \area_E),$ where $C_\eta$ is a constant (depending on $\CT \eta$).
 
 For $\lambda\in \mathbb C$, $\delta > 0$, a smooth non-negative function $\varphi$ with support in $\D(\lambda, \delta),$ $\varphi \le 1,$ and $\left \|\frac{\partial \varphi (z)}{\partial  \bar z} \right \|_\infty \le \frac{C_{45}}{\delta},$ using \eqref{AreaGammaEq}, we have
\[
\ \begin{aligned}
\ & \left | \int (z - \lambda)^n \CT (w\eta) \bar \partial \varphi d\area_E \right | \\
\ \le & \left | \int (z - \lambda)^n \CT ( w\eta) \bar \partial \varphi d\area \right | + \delta ^n \|\CT (w\eta)\| \| \bar \partial \varphi\| \area (\D(\lambda, \delta) \setminus E) \\ 
\ \le & \pi \delta ^n \eta(\D(\lambda, \delta)) + \delta ^{n-1}\area (\D(\lambda, \delta)\setminus \mathcal R)) \\
\ \le & C_{53} C_\eta \delta ^n \gamma (\D(\lambda, 2k_1\delta) \setminus \mathcal R).
\ \end{aligned}
\]
Now the proof follows from Theorem \ref{REEAlgTheorem}.
\end{proof}

\begin{corollary}\label{CTInRI2}
Let $\eta \in M_0^+(\C)$ and $\|\CT\eta\|_{L^\i(\C)} \le 1.$ If $\CT\eta \in R^\i(\overline E, \area_E),$ then $\eta (\mathcal R) = 0.$
\end{corollary}

\begin{proof}
Suppose that there exists a compact subset $B_0\subset \mathcal R$ such that $\eta (B_0) > 0.$
Then from Corollary \ref{ZeroACEta}, we see that $\gamma (B_0) > 0.$
By Proposition \ref{GammaPlusThm} (1)-(3), there exists $B_1 \subset B_0$ with $\eta (B_1) > 0$ and $N_2(\eta_{B_1}) < \i.$ Without loss of generality, we assume $N_2(\eta_{B_1})\le 1.$ Using Proposition \ref{GammaPlusThm} (4) and Theorem \ref{TolsaTheorem} (1), we have $\eta_{B_1} (B_2) \le C_{54} \gamma(B_2)$ for $B_2\subset B_1$ since $N_2(\eta_{B_2})\le 1.$ Therefore, applying Lemma \ref{CTMaxFunctFinite}, we may assume $\CT_*(g_n\area_E)\in L^\i (\eta_{B_1}).$
Using Proposition \ref{GammaPlusThm} (4), we find $0 \le w(z) \le 1$ supported on $B_1$ such that  $\eta (B_1) \le 2\int w d \eta_{B_1},$ and $\|\CT_\epsilon (w\eta_{B_1}) \| \le C_{55}.$
Thus, for a non-negative smooth function $\psi$ with compact support, $T_\psi (\CT (w\eta_{B_1})) = \CT (\psi w\eta_{B_1}) \in L^\i(\C).$ From Corollary \ref{CTInRI}, we see that $\CT (\psi w\eta_{B_1}) \in R^\i(\overline E, \area_E).$ Using the Lebesgue dominated convergence theorem and taking $\epsilon\rightarrow 0$ for
\[
\ \int \CT_\epsilon (g_n\area_E)\psi wd\eta_{B_1} = - \int \CT_\epsilon (\psi w\eta_{B_1}) g_nd\area_E,  
\]
we obtain
$\int \CT(g_n\area_E)\psi wd\eta_{B_1} = 0, ~ n\ge 1$,	
which implies $\CT(g_n\area_E)(z) = 0,~w\eta_{B_1}-a.a..$ This is a contradiction. The proof completes.
\end{proof}

\section{\textbf{Proof of Theorem \ref{InvTheorem}}}

For a dense subset $\Lambda = \{g_n\} \subset R(\overline E)^\perp \cap L^1(\area_E)$ and $N\ge 1,$ define
\[
\ \mathcal E_N = \left \{\lambda : ~\lim_{\epsilon \rightarrow 0} \mathcal C_\epsilon(g_n\area_E)(\lambda )\text{ exists, } \max_{1\le n\le N} |\mathcal C (g_n\area_E)(\lambda ) | \le \frac{1}{N} \right \}.
 \]

\begin{lemma}\label{ENEstimate}
There exists an absolute constant $C > 0$ such that
 \[
 \ \lim_{N\rightarrow\infty} \gamma(\D(\lambda, \delta)\cap \mathcal E _N) \le C\gamma(\D(\lambda, 2\delta)\cap \mathcal F ).
 \]
\end{lemma}

\begin{proof}
Set
 \[
 \ \epsilon_0 = \lim_{N\rightarrow\infty} \gamma(\D (\lambda, \delta)\cap \mathcal E_N) ( = \inf_{N\ge 1} \gamma(\D (\lambda, \delta)\cap \mathcal E_N)).
 \]
We assume $\epsilon_0 > 0$. The case that $\epsilon_0 = 0$ is trivial.   
From Lemma \ref{CTMaxFunctFinite}, there exists a Borel subset $F\subset \D (\lambda, \delta)$ such that
(a) $\gamma (\D (\lambda, \delta) \setminus F) < \frac{1}{2A_T} \epsilon_0$ and
(b) $\mathcal C_*(g_n\mu)(z) \le M_n < \infty$ for $z \in F.$ By Theorem \ref{TolsaTheorem} (2) and (a),
 \[
 \  \gamma(\D (\lambda, \delta)\cap \mathcal E_N) \le A_T\gamma(\D (\lambda, \delta) \cap F \cap \mathcal E_N) + \frac {1}{2} \epsilon_0.
 \]
 So $\gamma(\D (\lambda, \delta)\cap \mathcal E_N) \le 2A_T\gamma(\D (\lambda, \delta) \cap F \cap \mathcal E_N).$
From Lemma \ref{BBFunctLemma}, there exists  $\eta_N \in M_0^+(\D(\lambda, \delta)\cap F \cap \mathcal E_N)$ with $1$-linear growth and $ \| \mathcal C_\epsilon (\eta_N) \| _\infty \le 1$ such that $\gamma(\D(\lambda, \delta)\cap F \cap \mathcal E_N)\le C_5\|\eta_N\| \le C_5\delta.$ Applying \eqref{BBFunctLemmaEq1}, we get 
\begin{eqnarray}\label{R0Eq1}
\ \left | \int \CT \eta_N g_n d\area_E \right | = \left | \int \CT (g_n \area_E) d \eta_N \right | \le \dfrac{\|\eta_N\|}{N}
\end{eqnarray}

By passing to a subsequence if necessary, we may assume that $\eta_N \rightarrow \eta$ in $C(\overline{\D (\lambda, \delta)\cap F})^*$ weak-star topology. 
Since $\mathcal C_\epsilon (\eta_N)$ converges to $\mathcal C_\epsilon (\eta)$ in $L^\infty (\mathbb C)$ weak-star topology as $N\rightarrow \infty,$ we get $\|\mathcal C_\epsilon (\eta)\| \le 1.$ Clearly,
$\text{spt}(\eta) \subset \overline{F}$, $\eta$ is $1$-linear growth, and $\lim_{N\rightarrow \infty} \|\eta_N\| = \|\eta \|.$ 
  
By passing to a subsequence, we may assume that $\mathcal C(\eta_N)$ converges to $H(z)$ in $L^\infty (\mathbb C)$ weak-star topology. Let $\varphi$ be a smooth function with compact support. Then $\CT(\varphi\area)$ is continuous and we have
\[
\begin{aligned}
\ \int \varphi H d\area = & \lim_{N\rightarrow\i} \int \varphi \CT \eta_N d\area = - \lim_{N\rightarrow\i} \int \CT (\varphi\area) d\eta_N  \\
\ = &- \int \CT (\varphi\area) d\eta = \int \varphi \CT \eta d\area.
\end{aligned}
\]
Hence, $H = \CT \eta.$ Letting $N\rightarrow \i$ to \eqref{R0Eq1}, we conclude that 
$\int \CT \eta g_n d\area_E = 0,~n\ge 1.$ So, $\CT \eta \in R^\i (\overline E, \area_E)$ by the Hahn-Banach Theorem.

Applying Corollary \ref{CTInRI2}, we see that $\eta(\mathcal R) = 0.$ There is an open subset $O$ such that $\text{spt}(\eta)\cap \mathcal R \subset O$ and $\eta(O) \le \frac 12 \|\eta\|$. If we define $W:= \text{spt}(\eta)\setminus O$, then $W$ is compact, $W \subset \text{spt}(\eta)\cap\mathcal F $, and
 $\|\eta\| \le 2\eta(W).$
Hence, by \eqref{CTC2Eq} ,
 \[
 \ c^2(\eta|_W) \le c^2(\eta) \le C_{56} \|\eta\| \le 2C_{56}\|\eta|_W\|.  
 \]
Let $\eta_1 = \frac{\eta|_W}{\sqrt{2C_{56}}}.$ Then $c^2(\eta_1) \le 1.$ Applying  Theorem \ref{TolsaTheorem} (1) and Proposition \ref{GammaPlusThm} (3) \& (4), we conclude that 
 \[
 \ \|\eta|_W\| \le C_{57} \gamma (W) \le C_{57}\gamma(\text{spt}(\eta)\cap\mathcal F).
 \]
 The lemma is proved.
\end{proof} 

\begin{lemma} \label{REEIntegral}
Let $f\in L^\i (\area_E)$ and $\|f\|_{L^\i (\area_E)} \le 1.$ Suppose that for $\epsilon > 0,$ there exists $A_\epsilon \subset \mathcal R$ with $\gamma(A_\epsilon) < \epsilon$ and there exists $f_{\epsilon, N}\in R(K_{\epsilon, N}),$ where $\mathcal R _{\epsilon, N} =  \mathcal R  \setminus (A_\epsilon\cup \mathcal E_N)$ and $K_{\epsilon, N} = \overline{\mathcal R _{\epsilon, N}},$  such that $\|f_{\epsilon, N}\| \le 2$ and 
\begin{eqnarray}\label{MLemma1RhoF}
 \ f (z) = f_{\epsilon, N}(z), ~ \area_{\mathcal R _{\epsilon, N}}-a.a..
 \end{eqnarray}
Then for  $\varphi$ a smooth function with support in $\D(\lambda, \delta)$,
 \[
 \ \left |\int f(z) \bar \partial \varphi (z) d\area_E (z) \right | \le C_{58} \delta \|\bar \partial \varphi\| \gamma(\D(\lambda, 2\delta) \cap \mathcal F).
 \]
\end{lemma}

\begin{proof} 
We extend $f_{\epsilon, N}$ as a continuous function on $\mathbb C$ with $\|f_{\epsilon, N}\|_{\mathbb C} \le 2$. Let $\varphi$ be a smooth function with $\text{supp}(\varphi) \subset \D(\lambda, \delta)$. Using \cite[Theorem VIII.8.1]{gamelin} and  Theorem \ref{TolsaTheorem} (2), we have the following calculation:
\[
 \ \begin{aligned}
 \  &|(T_\varphi f_{\epsilon, N})'(\infty) | \\
 \ = &\dfrac{1}{\pi}\left | \int f_{\epsilon, N}(z) \bar \partial \varphi (z) d \area(z) \right | \\
 \ \le & C_{59} \delta\|\bar\partial \varphi\| \gamma(\D(\lambda, \delta)\setminus K_{\epsilon, N}) \\
 \ \le &C_{59} \delta\|\bar\partial \varphi\| \gamma(\D(\lambda, \delta)\cap(\mathcal F \cup A_\epsilon\cup\mathcal E_N)) \\
 \ \le & C_{60} \delta\|\bar\partial \varphi\| (\gamma(\D(\lambda, \delta)\cap\mathcal F ) + \gamma(A_\epsilon ) + \gamma(\D(\lambda, \delta)\cap \mathcal E_N))
 \end{aligned}
 \]
Notice that $f(z) = 0, ~ \area|_{\mathcal F }-a.a.$,  by \eqref{MLemma1RhoF}, \eqref{AreaGammaEq}, and Theorem \ref{TolsaTheorem} (2), we get,  
 \[
 \ \begin{aligned}
 \ &\left | \int f(z) \bar \partial \varphi (z) d \area_E(z) \right | \\
 \ \le & \left | \int (f(z) - f _{\epsilon, N}(z))\bar \partial \varphi (z) d \area(z) \right | + \left | \int f_{\epsilon, N}(z) \bar \partial \varphi (z) d \area(z) \right | \\
 \ \le & \int _{\mathcal R _{\epsilon, N}^c}|f(z) - f _{\epsilon, N}(z)||\bar \partial \varphi (z)| d \area(z) + \pi |(T_\varphi f_{\epsilon, N})'(\infty) | \\
 \ \le & C_{61} \| \bar \partial \varphi\| \area(\D(\lambda, \delta)\cap \mathcal R _{\epsilon, N}^c)  + \pi |(T_\varphi f_{\epsilon, N})'(\infty) | \\
 \ \le & C_{62} \delta \| \bar \partial \varphi\| \gamma (\D(\lambda, \delta)\cap \mathcal R _{\epsilon, N}^c)  + \pi |(T_\varphi f_{\epsilon, N})'(\infty) | \\
 \ \le & C_{63} \delta\|\bar\partial \varphi\| (\gamma(\D(\lambda, \delta)\cap\mathcal F ) + \gamma(A_\epsilon ) + \gamma(\D(\lambda, \delta)\cap \mathcal E_N)).
 \end{aligned}
 \]
Thus, using Lemma \ref{ENEstimate} and taking $\epsilon\rightarrow 0$ and $N\rightarrow \infty$, we get
 \[
 \ \left | \int f(z) \bar \partial \varphi (z) d \area_E(z) \right | \le C_{64} \delta\|\bar\partial \varphi\| \gamma(\D(\lambda, 2\delta)\cap\mathcal F ).
 \]
This completes
the proof.
\end{proof}

\begin{lemma} \label{REEIntegral2}
Let $f\in R^\i (\overline E,\area_E)$ and $\|f\|_{L^\i (\area_E)} \le 1.$ Let $\{r_m\}\subset \text{Rat}(\overline E)$ such that $\sup_m\|r_m\|_{L^\i (\area_E)} < \i$ and $r_m(z) \rightarrow f(z),~\area_E-a.a..$ Then the following properties hold:
\newline
(1) There exists $\mathcal Q_f$ with $\gamma (\mathcal Q_f) = 0$ such that $f$ can be extended to $\mathcal R\setminus \mathcal Q_f,$ $f$ is $\gamma$-continuous at each point of $\mathcal R\setminus \mathcal Q_f,$ and $f(z)\CT(g_n\area_E)(z) = \CT(fg_n\area_E)(z)$ for $z\in \mathcal R\setminus \mathcal Q_f$ and $n\ge 1.$
\newline
(2) For $\epsilon > 0,$ there exists $A_\epsilon \subset \mathcal R$ with $\gamma(A_\epsilon) < \epsilon$ and a subsequence $\{r_{m_k}\}$ such that $r_{m_k}$ uniformly converges to $f$ on $\mathcal R _{\epsilon, N} :=  \mathcal R  \setminus (A_\epsilon\cup \mathcal E_N).$ Consequently, there exists $f_{\epsilon, N}\in R(K _{\epsilon, N}),$ where $K _{\epsilon, N} = \overline{\mathcal R _{\epsilon, N}},$ such that $f(z) = f_{\epsilon, N}(z)$ for $z\in \mathcal R _{\epsilon, N}.$
\end{lemma}

\begin{proof}
(1) For $r\in \text{Rat}(\overline E),$ $\frac{r(z)-r(\lambda)}{z-\lambda} \in \text{Rat}(\overline E).$ Hence, by Corollary \ref{ZeroAC},
\begin{eqnarray} \label{REEIntegral2Eq1}
\ r_m(z)\CT(g_n\area_E)(z) = \CT(r_mg_n\area_E)(z),~\gamma-a.a.
\end{eqnarray}
 for $n\ge 1.$ Therefore, $f(z)\CT(g_n\area_E)(z) = \CT(fg_n\area_E)(z),~\area_E-a,a,$ for $n\ge 1.$ Using Lemma \ref{RNDecom2} and Corollary \ref{ZeroAC}, we conclude that there exists $\mathcal Q_f$ with $\gamma (\mathcal Q_f) = 0$ such that for $z\in \mathcal R\setminus \mathcal Q_f,$ $\Theta_{g_n\area_E}(z) = 0,$ $\lim_{\epsilon\rightarrow 0}\CT_\epsilon (g_n\area_E)(z)$	exists, $\Theta_{fg_n\area_E}(z) = 0,$ and $\lim_{\epsilon\rightarrow 0}\CT_\epsilon (fg_n\area_E)(z)$ exists. It is straightforward to verify (1) from Lemma \ref{GCProp} and Lemma \ref{CauchyTLemma}. Therefore, for $n\ge 1,$
 \begin{eqnarray} \label{REEIntegral2Eq2}
\ f(z)\CT(g_n\area_E)(z) = \CT(fg_n\area_E)(z),~\gamma-a.a..
\end{eqnarray}

(2): Using Lemma \ref{CTUniformC}, we find $A^1_\epsilon$ and a subsequence $\{r_{m,1}\}$ of $\{r_{m}\}$ such that $\gamma(A^1_\epsilon) < \frac{\epsilon}{2A_T}$ and $\{\mathcal C(r_{m,1}g_1\area_E)\}$ uniformly converges to $\mathcal C(fg_1\area_E)$ on $\C \setminus A^1_\epsilon$. Then we find $A^2_\epsilon$ and a subsequence $\{r_{m,2}\}$ of $\{r_{m,1}\}$ such that $\gamma(A^2_\epsilon) < \frac{\epsilon}{2^2A_T}$ and $\{\mathcal C(r_{m,2}g_2\area_E)\}$ uniformly converges to $\mathcal C(fg_2\area_E)$ on $\C \setminus A^2_\epsilon$. Therefore, we have a subsequence $\{r_{m,m}\}$ such that $\{\mathcal C(r_{m,m}g_n\area_E)\}$ uniformly converges to $\mathcal C(fg_n\area_E)$ on $\C \setminus A_\epsilon$ for all $n \ge 1$, where $A_\epsilon = \mathcal Q_f \cup \cup_n A^n_\epsilon$ and $\gamma(A_\epsilon) < \epsilon$ by Theorem \ref{TolsaTheorem} (2).  
From \eqref{REEIntegral2Eq1} and \eqref{REEIntegral2Eq2}, we infer that $\{r_{m,m}\}$ uniformly tends to $f$ on $\mathcal R _{\epsilon, N}.$ Thus, $\{r_{m,m}\}$ uniformly tends to $f_{\epsilon, N}\in R(K_{\epsilon, N})$ on $K_{\epsilon, N}.$  
\end{proof}

Now we are ready to finish the proof of Theorem \ref{InvTheorem} as the following.

\begin{proof} (Theorem \ref{InvTheorem}):  
$\Rightarrow$: If $f\in R^\i(\overline E, \area_E)$ is invertible, then there exists  $f_0\in R^\i(\overline E, \area_E)$ such that $ff_0 = 1.$ Therefore, $|f(z)| \ge \frac{1}{\|f_0\|},~\area_E-a.a..$ 

$\Leftarrow$: Suppose that $f\in R^\i(\overline E, \area_E)$ and there exists $\epsilon_f > 0$ such that $|f(z)| \ge \epsilon_f,~\area_E-a.a..$ Using Lemma \ref{REEIntegral2} (2), for $\epsilon > 0,$ there exists $f_{\epsilon, N}\in R(K _{\epsilon, N})$ such that $f (z) = f_{\epsilon, N}(z)$ for $z\in \mathcal R_{\epsilon, N}.$ Hence, $|f_{\epsilon, N}(z)| \ge \epsilon_f,~z\in K _{\epsilon, N}.$ Thus, $f_{\epsilon, N}$ is invertible in $R(K _{\epsilon, N}).$ Using Lemma \ref{REEIntegral} for $\frac{\epsilon_f}{f(z)}$ and $\frac{\epsilon_f}{f_{\epsilon, N}(z)}$, we conclude that for  $\varphi$ a smooth function with support in $\D(\lambda, \delta)$, we have
 \[
 \ \left |\int \dfrac{\epsilon_f}{f(z)} \bar \partial \varphi (z) d\area_E (z) \right | \le C_{65} \delta \|\bar \partial \varphi\| \gamma(\D(\lambda, 2\delta) \cap \mathcal F).
 \]
 Replacing $\varphi$ by $(z-\lambda)^n\varphi$ and applying Theorem \ref{REEAlgTheorem}, we conclude that $\frac 1f \in R^\i(\overline E, \area_E).$
\end{proof}

\begin{remark}
 Using Lemma \ref{REEIntegral2} and Lemma \ref{REEIntegral}, we see that the constant $C_f$ in Theorem \ref{REEAlgTheorem} (2) can be chosen as an absolute constant (independent of $f$).  	
\end{remark}

\bigskip

{\bf Acknowledgments.} 
The author would like to thank Professor John M\raise.45ex\hbox{c}Carthy for carefully reading through the manuscript and providing many useful comments. The author would also like to thank the referee for carefully reading the manuscript and providing helpful comments.

\bigskip

\bibliographystyle{amsplain}

\end{document}